\newif\ifArXiv
\pdfoutput=1 \ArXivtrue % enable this line FOR ARXIV

\documentclass[a4paper]{article}
\usepackage{amsthm,amssymb,amsmath,enumerate,graphicx}

\newtheorem{THM}{Theorem}[section]
\newtheorem{LEM}[THM]{Lemma}

\newtheorem{COR}[THM]{Corollary}

\newcommand\abs[1]{\left\lvert #1\right\rvert}
\def\ord(#1){\abs{#1}}
\def\shift(#1)(#2){\!\!\downarrow\!{}^{#1}_{\raise .01ex\vbox to 0pt{\vss\hbox{$\scriptstyle #2$}}}\,}
\def\ucl(#1){\lfloor #1 \rfloor}% up-closure
\def\dcl(#1){\lceil #1 \rceil}% down-closure
\def\specrel#1#2{\mathrel{\mathop{\kern0pt #1}\limits_{#2}}}
\def\starorder#1{\langle#1\rangle}

\newcommand\A{\mathcal A}
\newcommand\B{\mathcal B}
\newcommand\T{\mathcal T}
\newcommand\V{\mathcal V}
\renewcommand\P{\mathcal P}

\def\F{\mathcal F}

\def\lowfwd #1#2#3{{\mathop{\kern0pt #1}\limits^{\kern#2pt\raise.#3ex
\vbox to 0pt{\hbox{$\scriptscriptstyle\rightarrow$}\vss}}}}
\def\lowbkwd #1#2#3{{\mathop{\kern0pt #1}\limits^{\kern#2pt\raise.#3ex
\vbox to 0pt{\hbox{$\scriptscriptstyle\leftarrow$}\vss}}}}

\def\ve{\kern-1pt\lowfwd e{1.5}2\kern-1pt}

\def\vedash{{\mathop{\kern0pt e\lower.5pt\hbox{${}% logically \v(e')
     \scriptstyle'$}}\limits^{\kern0pt\raise.02ex
     \vbox to 0pt{\hbox{$\scriptscriptstyle\rightarrow$}\vss}}}}
\def\evdash{{\mathop{\kern0pt e\lower.5pt\hbox{${}% logically \v(e')
     \scriptstyle'$}}\limits^{\kern0pt\raise.02ex
     \vbox to 0pt{\hbox{$\scriptscriptstyle\leftarrow$}\vss}}}}
\def\vf{\kern-1pt\lowfwd f{1.5}2\kern-1pt}

\def\vr{\lowfwd r{1.5}2}
\def\rv{\lowbkwd r02}

\def\vrdash{{\mathop{\kern0pt r\lower.5pt\hbox{${}% logically \v(e')
     \scriptstyle'$}}\limits^{\kern0pt\raise.02ex
     \vbox to 0pt{\hbox{$\scriptscriptstyle\rightarrow$}\vss}}}}
\def\rvdash{{\mathop{\kern0pt r\lower.5pt\hbox{${}% logically \v(e')
     \scriptstyle'$}}\limits^{\kern0pt\raise.02ex
     \vbox to 0pt{\hbox{$\scriptscriptstyle\leftarrow$}\vss}}}}
\def\vrddash{{\mathop{\kern0pt r\lower.5pt\hbox{${}% logically \v(e')
     \scriptstyle''$}}\limits^{\kern0pt\raise.02ex
     \vbox to 0pt{\hbox{$\scriptscriptstyle\rightarrow$}\vss}}}}
\def\rvddash{{\mathop{\kern0pt r\lower.5pt\hbox{${}% logically \v(e')
     \scriptstyle''$}}\limits^{\kern0pt\raise.02ex
     \vbox to 0pt{\hbox{$\scriptscriptstyle\leftarrow$}\vss}}}}
\def\vrone{\lowfwd {r_1}02}

\def\vs{\lowfwd s{1.5}1}
\def\sv{\lowbkwd s{1.5}1}
\def\vso{\lowfwd {s_0}11}
\def\svo{\lowbkwd {s_0}02}

\def\vsidash{{\mathop{\kern0pt s_i\kern-3.5pt\lower.3pt\hbox{${}
     \scriptstyle'$}}\limits^{\kern0pt\raise.02ex
     \vbox to 0pt{\hbox{$\scriptscriptstyle\rightarrow$}\vss}}}}
\def\vS{{\vec S}} % {{\fwd S3}}
\def\vSr{{\vec S}_{\raise.1ex\vbox to 0pt{\vss\hbox{$\scriptstyle\ge\vr$}}}}
\def\vSrone{{\vec S}_{\raise.1ex\vbox to 0pt{\vss\hbox{$\scriptstyle\ge\vrone$}}}}
\def\vSdash{{\mathop{\kern0pt S\lower-1pt\hbox{${}% logically \v(e')
     \scriptstyle'$}}\limits^{\kern2pt\raise.1ex
     \vbox to 0pt{\hbox{$\scriptscriptstyle\rightarrow$}\vss}}}}
\def\vsdash{{\mathop{\kern0pt s\lower.5pt\hbox{${}% logically \v(e')
     \scriptstyle'$}}\limits^{\kern0pt\raise.02ex
     \vbox to 0pt{\hbox{$\scriptscriptstyle\rightarrow$}\vss}}}}
\def\svdash{{\mathop{\kern0pt s\lower.5pt\hbox{${}% logically \v(e')
     \scriptstyle'$}}\limits^{\kern0pt\raise.02ex
     \vbox to 0pt{\hbox{$\scriptscriptstyle\leftarrow$}\vss}}}}
\def\vt{\lowfwd t{1.5}1}
\def\tv{\lowbkwd t{1.5}1}
\def\vU{{\vec U}} % {{\fwd S3}}

\def\sub{\subseteq}
\def\supe{\supseteq}
\def\sm{\smallsetminus}
\def\td{tree-decom\-po\-si\-tion}
\def\pd{path-decom\-po\-si\-tion}
\def\bd{branch-decom\-po\-si\-tion}
\def\?#1{\vadjust{\vbox to 0pt{\vss\vskip-8pt\leftline{%
     \llap{\hbox{\vbox{\pretolerance=-1
     \doublehyphendemerits=0\finalhyphendemerits=0
     \hsize16truemm\tolerance=10000\small
     \lineskip=0pt\lineskiplimit=0pt
     \rightskip=0pt plus16truemm\baselineskip8pt\noindent
     \hskip0pt        %(without this, the first word is never hyphenated!)
     #1\endgraf}\hskip7truemm}}}\vss}}}
\def\COMB#1{\vadjust{\vbox to 0pt{\vss\vskip-8pt\leftline{%  ***DISABLED BY NEW DEF \BELOW***
     \llap{\hbox{\vbox{\pretolerance=-1
     \doublehyphendemerits=0\finalhyphendemerits=0
     \hsize16truemm\tolerance=10000\small
     \lineskip=0pt\lineskiplimit=0pt
     \rightskip=0pt plus16truemm\baselineskip8pt\noindent
     \hskip0pt        %(without this, the first word is never hyphenated!)
     #1\endgraf}\hskip7truemm}}}\vss}}}

\newcommand\COMMENT[1]{}
\def\COMB#1{}

\title{Tangle-tree duality:\\ in graphs, matroids and beyond%
   \ifArXiv\footnote{This is an extended version of~\cite{TangleTreeGraphsMatroids} available only in preprint form.}\fi}
\author{Reinhard Diestel\\
  Mathematisches Seminar, Universit\"at Hamburg
� \and�
� Sang-il Oum%
\thanks{Supported by the National Research Foundation of Korea (NRF) grant funded by the Korea government (MSIT) (No. NRF-2017R1A2B4005020).}
�\\
� KAIST, Daejeon, 34141 South Korea}

\begin{document}
\abovedisplayshortskip=-3pt plus3pt
\belowdisplayshortskip=6pt

\maketitle

\begin{abstract}\noindent
  We apply a recent tangle-tree duality theorem in abstract separation systems to derive tangle-tree-type duality theorems for width-parameters in graphs and matroids. We further derive a duality theorem for the existence of clusters in large data sets.

Our applications to graphs include new, tangle-type, duality theorems for tree-width, path-width, and \td s of small adhesion. Conversely, we show that carving width is dual to edge-tangles.%
   \COMB{{\bf Referees:}\\[6pt]  Changes resulting from your comments (other\,than typos) are indicated in the margin, like this comment.\\[6pt] See separate responses for suggestions we did not implement.}
   For matroids we obtain a tangle-type duality theorem for tree-width.

Our results can also be used to derive short proofs of all the classical duality theorems for width para\-me\-ters in graph minor theory, such as path-width, tree-width, branch-width and rank-width.%
  \COMMENT{}
   \end{abstract}

\section{Introduction}\label{sec:intro}

There are a number of theorems in the structure theory of sparse graphs that assert a duality between high connectivity present somewhere in the graph and an overall tree structure. For example, a graph has small tree-width if and only if it contains no large-order bramble. Amini, Lyaudet, Mazoit, Nisse and Thomass\'e~\cite{MazoitPartition, MazoitPushing}%
   \COMB{Lyaudet and \cite{MazoitPushing} added (Referee~2)}
   generalized the notion of a bramble to give similar duality theorems for other width parameters, including branch-width, rank-width and matroid tree-width. The highly cohesive substructures, or {\em HCSs}, dual to low width in all these cases are what we call {\em concrete\/}~HCSs: like brambles, they are sets of edges that hang together in a certain specified way.

In~\cite{TangleTreeAbstract} we considered another type of HCSs for graphs and matroids, which we call {\em abstract\/} HCS. These are modelled on the notion of a {\em tangle\/} introduced by Robertson and Seymour~\cite{GMX} for the proof of the graph minor theorem. They are orientations of all the separations of a graph or matroid, up to some given order, that are `consistent' in a way specified by a set~$\F$. This $\F$ can be varied to give different notions of consistency, leading to different notions of $\F$-{\em tangles\/}. 

In~\cite[Theorem 4.3]{TangleTreeAbstract} we proved a general duality theorem for $\F$-tangles in an abstract setting that includes, but goes considerably beyond, graphs and matroids. Applied to graphs and matroids, the theorem says that a graph or matroid not containing an $\F$-tangle has a certain type of tree structure, the type depending on the choice of~$\F$. Conversely, this tree structure clearly precludes the existence of an $\F$-tangle, and thus provides an easily checked certificate for their possible nonexistence.%

Classical tangles of graphs are examples of $\F$-tangles for a suitable choice of~$\F$, so our abstract duality theorem from~\cite{TangleTreeAbstract} implies a duality theorem for classical tangles. The tree structures we obtain as witnesses for the non-existence of such tangles differ slightly from the {\em branch-decompositions\/} of graphs featured in the tangle-tree duality theorem of Robertson and Seymour~\cite{GMX}. While our tree structures are squarely based on graph separations, their branch-decompositions are, in spirit, translated from decompositions of the graph's cycle matroid, which takes a toll for low-order tangles where our result is a little cleaner. 

Conversely, we show that $\F$-tangles can be used to witness large tree-width or path-width of graphs, giving new duality theorems for these width-parameters. Like the classical, bramble-based, tree-width duality theorem of Seymour and Thomas~\cite{ST1993GraphSearching}, and their duality theorem for path-width with Bienstock and Robertson~\cite{QuicklyExcludingForest}, our duality theorem is exact. Our theorems easily imply theirs, but not conversely. By tweaking~$\F$, we can obtain tailor-made duality theorems also for particular kinds of \td s as desired, such as those of some specified adhesion.

Matroid tree-width was introduced only more recently, by Hlin{\v e}n{\'y} and Whittle~\cite{HlinenyWhittle}, and we shall obtain a tangle-type duality theorem for this too.

Another main result in this paper is a general width-duality theorem for $\F$-tangles of bipartitions of a set. Applied to bipartitions of the ground set of a matroid this implies the duality theorem for matroid tangles derived from~\cite{GMX} by Geelen, Gerards and Whittle~\cite{BranchDecMatroids}. Applied to bipartitions of the vertex set of a graph it implies a duality theorem for rank-width~\cite{rankwidth}. Applied again to bipartitions of the vertex set of a graph, but with a different~$\F$, it yields a duality theorem for the {\em edge-tangles\/} introduced recently by Liu~\cite{LiuEdgeTangles} as a tool for proving an Erd\H os-P\'osa-type theorem for edge-disjoint immersions of graphs. Interestingly, it turns out that the corresponding tree structures were known before: they are the {\em carvings\/} studied by Seymour and Thomas~\cite{ratcatcher}, but this duality appears to have gone unnoticed.

Our $\F$-tangle-tree duality theorem for set partitions is in fact a special case of a duality theorem for {\em set separations\/}: pairs of subsets whose union is a given set, but which may overlap (unlike the sets in a bipartition, which are disjoint).

Indeed, we first proved the abstract duality theorem of~\cite{TangleTreeAbstract} that we keep applying here for the special case of set separations. This version already implied all the results mentioned so far. However we then noticed that we needed much less to express, and to prove, this duality theorem. As a result, we can now use `abstract' $\F$-tangles to describe clusters not only in graphs and matroids but in very different contexts too, such as large data sets in various applications, and derive duality theorems casting the set into a tree structure whenever it contains no such cluster. As an example to illustrate this, we shall derive an $\F$-tangle-tree duality theorem for image analysis, which provides fast-checkable witnesses to the non-existence of a coherent region of an image, which could be used for a rigorous proof of the low quality of a picture, e.g.\ after transmission through a noisy channel.%
   \footnote{We make no claim here as to how fast such a witness might be computable, only that checking it will be fast. This is because such a check involves only linearly many separations.\COMB{Footnote added (Referee~3)}%
   \COMMENT{}
   Exploring the $\F$-tangle-tree theorem from an algorithmic point of view is a problem we would indeed like to see tackled. Any good solution is likely to depend on~$\F$, though, and thus on the concrete application considered.}

Let us briefly explain these `abstract' tangles. The oriented separations in a graph or matroid are partially ordered in a natural way, as $(A,B)\le (C,D)$ whenever $A\sub C$ and $B\supe D$. This partial ordering is inverted by the involution $(A,B)\mapsto (B,A)$. Following~\cite{AbstractSepSys}, let us call any poset $(\vS,\le)$ with an order-reversing involution $\vs\mapsto\sv$ an {\em abstract separation system\/}. If this poset is a lattice, we call it a {\em separation universe\/}. The set of all the separations of a graph or matroid, for example, is a universe, while the set of all separation of order~$<k$ for some integer~$k$ is a separation system that may fail to be a universe.

All the necessary ingredients of $\F$-tangles in graphs, and of their dual tree structures, can be expressed in terms of~$(\vS,\le)$. Indeed, two separations are nested if and only if they have orientations that are comparable under~$\le$. And the consistency requirement for classical tangles is, essentially, that if $\vr$ and~$\vs$ `lie in' the tangle (i.e., if the tangle orients $r$ as~$\vr$ and $s$ as~$\vs$) then so does their supremum $\vr\lor\vs$, if it is in~$\vS$. It turned out that this was not a special case: we could express the entire duality theorem and its proof in this abstract setting.%
   \COMMENT{}
   Put more pointedly, we never need that our separations actually `separate' anything: all we ever use is how they relate to each other in terms of~$(\vS,\le)$.

For example, the bipartitions of a (large data) set~$D$ form a separation universe: they are partially ordered by inclusion of their sides, and the involution of flipping the sides of the bipartition inverts this ordering. Depending on the application, some ways of cutting the data set in two will be more natural than others, which gives rise to a cost function on these separations of~$D$.%
   \footnote{The bipartitions of $D$ considered could be chosen according to some property that some elements of $D$ have and others lack. We could also allow the two sides to overlap where this property is unclear: then we no longer have bipartitions, but still set separations.}
   Taking this cost of a separation as its `order' then gives rise to tangles:%
   \COMMENT{}
   abstract HCSs signifying clusters. Unlike clusters defined by simply specifying a subset of~$D$, clusters defined by tangles will be fuzzy%
   \COMB{(Referee~3)}
   in terms of which data they `contain'~-- much like clusters in real-world applications.

If the cost function on the separations of our data set is submodular~-- which in practice may not be a severe restriction~-- the abstract duality theorem from~\cite{TangleTreeAbstract} can be applied to these tangles. For every integer~$k$, our application of this theorem will either find a cluster of order at least~$k$ or produce a nested `tree' set of bipartitions, all of order~${<k}$, which together witness that no such cluster exists. An example from image analysis, with a cost function chosen so that the clusters become the visible regions in a picture, is given in~\cite{MonaLisa}. This information could be used, for example, to assess the quality of an image, eg.\ after sending it through a noisy channel.

\medbreak

Our paper is organized as follows. We begin in Section~\ref{sec:def} with a brief description of abstract separation systems: just enough to state in Section~\ref{sec:AbstractDuality}, as Theorem~\ref{thm:strong}, the tangle-tree duality theorem of~\cite{TangleTreeAbstract} that we shall be applying throughout.\looseness=-1

In Section~\ref{sec:bwd} we prove our duality theorem for classical tangles as introduced by Robertson and Seymour~\cite{GMX}, and indicate how to derive their tangle-branchwidth duality  theorem if desired.

In Section~\ref{sec:bwdsubmodular} we apply Theorem~\ref{thm:strong} to set separations with a submodular order function. By specifying this order function we obtain duality theorems for rank-width, edge-tangles, and carving-width in graphs, for tangles in matroids and, as an example of an application beyond graphs and matroids, for clusters in large data sets such as coherent features in pixellated images.

In Sections \ref{sec:twd} and~\ref{sec:pathwidth} we obtain our new duality theorems for tree-width and path-width, and show how to derive from these the existing but different duality theorems for these parameters.

In Section~\ref{sec:matroid} we prove our duality theorem for matroid tree-width. In Section~\ref{sec:adhesion} we derive duality theorems for \td s of bounded adhesion.

\ifArXiv In Section~\ref{sec:partitionsub}, finally, \else The ArXiv version of this paper~\cite{TangleTreeGraphsMatroidsArXiv} has a further section in which \fi we show how our duality theorem for abstract tangles, Theorem~\ref{thm:strong}, implies the duality theorem for abstract brambles of Amini, Maz\-oit, Nisse, and Thomass\'e~\cite{MazoitPartition} under a mild additional assumption, which holds in all their applications.

\section{Abstract separation systems}\label{sec:def}

In this section we describe the basic features of abstract separation systems~\cite{AbstractSepSys}~-- just enough to state the main duality theorem from~\cite{TangleTreeAbstract} in Section~\ref{sec:AbstractDuality}, and thus make this paper self-contained.

A \emph{separation of a set} $V$ is a set $\{A,B\}$ such that $A\cup B=V$. The ordered pairs $(A,B)$ and $(B,A)$ are its {\it orientations\/}. The {\em oriented separations\/} of~$V$ are the orientations of its separations. Mapping every oriented separation $(A,B)$ to its {\it inverse\/} $(B,A)$ is an involution%
   \COMMENT{}
   that reverses the partial ordering \[(A,B)\le (C,D) :\Leftrightarrow A\subseteq C \text{ and } B\supseteq D.\]
Note that this is equivalent to $(D,C)\le (B,A)$. Informally, we think of $(A,B)$ as \emph{pointing towards}~$B$ and \emph{away from}~$A$. Similarly, if $(A,B)\le (C,D)$, then $(A,B)$ \emph{points towards} $\{C,D\}$ and its orientations, while $(C,D)$ \emph{points away from} $\{A,B\}$ and its orientations.%
   \COMB{`and its orientations' added (Referee~3)}

Generalizing these properties of separations of sets, we now give an axio\-matic definition of `abstract' separations. A {\em separation system\/} $(\vS,\le\,,\!{}^*)$ is a partially ordered set $\vS$ with an order-reversing involution~*. Its elements are called {\em oriented separations\/}. When a given element of $\vS$ is denoted as~$\vs$, its {\em inverse\/}~$\vs^*$ will be denoted as~$\sv$, and vice versa. The assumption that * be {\em order-reversing\/} means that, for all $\vr,\vs\in\vS$,
\begin{equation}\label{invcomp}
\vr\le\vs\ \Leftrightarrow\ \rv\ge\sv.
\end{equation}

A {\em separation\/} is a set of the form $\{\vs,\sv\}$, and then denoted by~$s$. We call $\vs$ and~$\sv$ the {\em orientations\/} of~$s$. The set of all such sets $\{\vs,\sv\}\sub\vS$ will be denoted by~$S$. If $\vs=\sv$, we call both $\vs$ and $s$ {\em degenerate\/}.

When a separation is introduced ahead of its elements and denoted by a single letter~$s$, its elements will then be denoted as $\vs$ and~$\sv$.%
   \footnote{It is meaningless here to ask which is which: neither $\vs$ nor $\sv$ is a well-defined object just given~$s$. But given one of them, both the other and $s$ will be well defined. They may be degenerate, in which case $s = \{\vs\} = \{\sv\}$.}
   Given a set $S'\sub S$ of separations, we write $\vSdash := \bigcup S'\sub\vS$%
   \COMMENT{}
   for the set of all the orientations of its elements. With the ordering and involution induced from~$\vS$, this is again a separation system.%
   \footnote{For $S'=S$, our definition of $\vSdash$ is consistent with the existing meaning of~$\vS$.%
   \COMMENT{}
   When we refer to oriented separations using explicit notation that indicates orientation, such as $\vs$ or $(A,B)$, we sometimes leave out the word `oriented' to improve the flow of words. Thus, when we speak of a `separation $(A,B)$', this will in fact be an oriented separation.}

Separations of sets, and their orientations, are clearly an instance of this if we identify $\{A,B\}$ with $\{(A,B),(B,A)\}$.

If a separation system~$(\vU,\le\,,\!{}^*)$ is a lattice, i.e., if there are binary operations $\vee$ and~$\wedge$ on $\vU$ such that $\vr\vee\vs$ is the supremum and $\vr\wedge\vs$ the infimum of $\vr$ and~$\vs$ in~$\vU$, we call $(\vU,\le\,,\!{}^*,\vee,\wedge)$ a {\em universe\/} of (oriented) separations. By~\eqref{invcomp}, it satisfies De~Mor\-gan's law:
\begin{equation}\label{deMorgan}
   (\vr\vee\vs)^* =\> \rv\wedge\sv.
\end{equation}%
   \COMMENT{}
   A separation system $\vS\sub\vU$, with its ordering and involution induced from~$\vU$, is {\em submodular\/} if for all $\vr,\vs\in\vS$ at least one of $\vr\land\vs$ and $\vr\lor\vs$ also lies in~$\vS$.

The oriented separations of a set~$V$ form such a universe: if $\vr = (A,B)$ and $\vs = (C,D)$, say, then $\vr\vee\vs := (A\cup C, B\cap D)$ and $\vr\wedge\vs := (A\cap C, B\cup D)$ are again oriented separations of~$V\!$, and are the supremum and infimum of $\vr$ and~$\vs$.%
   \COMMENT{}
   Similarly, the oriented separations of a graph form a universe. Its oriented separations of order~$<k$ for some fixed~$k$, however, form a separation system~$S_k$ inside this universe that may not itself be a universe with respect to $\vee$~and~$\wedge$ as defined above.%
   \COMMENT{}
   However, it is easy to check that $S_k$~is submodular.

A separation $\vr\in\vS$ is {\em trivial in~$\vS$\/}, and $\rv$ is {\em co-trivial\/}, if there exists $s \in S$ such that $\vr < \vs$ as well as $\vr < \sv$.%
   \COMMENT{}
   Note that if $\vr$ is trivial in~$\vS$ then so is every $\vrdash \le \vr$. If $\vr$ is trivial, witnessed by~$\vs$, then $\vr < \vs < \rv$ by~\eqref{invcomp}. Separations~$\vs$ such that $\vs\le\sv$, trivial or not, will be called {\em small\/}.

For example, the oriented separations of a set~$V\!$ that are trivial in the universe of all the oriented separations of~$V\!$ are those of the form $\vr = (A,B)$ with $A\sub C\cap D$ and $B\supe C\cup D = V$ for some separation $s = \{C,D\}\ne r$ of~$V\!$. The small separations $(A,B)$ of $V$ are all those with $B=V$.

Two separations $r,s$ are {\em nested\/} if they have comparable orientations; otherwise they \emph{cross}. Two oriented separations $\vr,\vs$ are {\em nested\/} if $r$ and~$s$ are nested.%
   \footnote{Terms introduced for unoriented separations may be used informally for oriented separations too if the meaning is obvious, and vice versa.}%
   \COMMENT{}
   We say that $\vr$ {\em points towards\/}~$s$, and $\rv$ {\em points away from\/}~$s$, if $\vr\le\vs$ or $\vr\le\sv$. Then two nested oriented separations are either comparable, or point towards each other, or point away from each other. A~set of separations is {\em nested\/} if every two of its elements are nested.

A set $O\sub \vS$ of oriented separations is {\em antisymmetric\/} if it does not contain the inverse of any of its nondegenerate elements. It is \emph{consistent} if there are no distinct $r,s\in S$ with orientations $\vr < \vs$ such that $\rv,\vs\in O$.%
   \COMMENT{}
   (Informally: if it does not contain orientations of distinct separations that point away from each other.)
   An \emph{orientation} of~${S}$ is a maximal antisymmetric subset of~$\vS$: a subset that contains for every $s\in{S}$ exactly one of its orientations $\vs,\sv$.

Every consistent orientation of $S$ contains all separations $\vr$ that are trivial in~$\vS$, because it cannot contain their inverse~$\rv$: if the triviality of $\vr$ is witnessed by $s\in S$, say, then $\rv$ would be inconsistent with both $\vs$ and~$\sv$.

Given a set~$\F$,%
   \COMMENT{}
   a~consistent orientation of~$S$ is an $\F$-{\em tangle\/}%
   \COMB{Footnote added (Referee~2)}%
   \footnote{The tangles introduced by Robertson and Seymour~\cite{GMX} for graphs are, essentially, the $\T_k$-tangles for the set $\T_k$ of triples of oriented separations $(A,B)$ of order less than some fixed~$k$ whose three `small' sides~$A$ together cover the graph. See Section~\ref{sec:bwd} for details.}
   if it avoids~$\F$, i.e., has no subset $F\in\F$. We think of~$\F$ as a collection of `forbidden' subsets of~$\vS$. Avoiding~$\F$ adds another degree of consistency to an already formally consistent orientation of~$S$, one that can be tailored to specific applications by designing~$\F$ in different ways. The idea is always that the oriented separations in a set $F\in\F$ collectively point to an area (of the ground set or structure which the separations in $S$ are thought to `separate') that is too small to accommodate some particular type of highly cohesive substructure.

A set $\sigma$ of nondegenerate oriented separations, possibly empty, is a \emph{star of separations} if they point towards each other: if $\vr\le\sv$ for all distinct $\vr,\vs\in\sigma$ (Fig.~\ref{fig:star}). Stars of separations are clearly nested. They are also consistent: if $\rv,\vs$ lie in the same star we cannot have $\vr < \vs$, since also $\vs\le \vr$ by the star property.
A~star~$\sigma$ need not be antisymmetric;%
   \COMMENT{}
   but if $\{\vs,\sv\}\sub\sigma$, then any other $\vr\in\sigma$ will be trivial.%
   \COMMENT{}

   \begin{figure}[htpb]
   \centering%\vskip-6pt
   	  \includegraphics{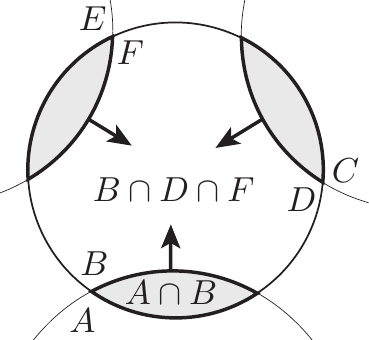}
   	  \caption{The separations $(A,B),(C,D),(E,F)$ form a 3-star}
   \label{fig:star}%\vskip-6pt
   \end{figure}

Let $S$ be a set of separations.  An \emph{${S}$-tree\/} is a pair $(T,\alpha)$ of a tree%
   \footnote{Trees have at least one node%
   \COMMENT{}%
   ~\cite{DiestelBook16}.}~$T$ and a function $\alpha\colon\vec E(T)\to \vS$ from the set
 $$\vec E(T) := \{\, (x,y) : \{x,y\}\in E(T)\,\}$$
 of the \emph{orientations of} its {\em edges} to~$\vS$ such that, for every edge $xy$ of~$T$, if $\alpha(x,y)=\vs$ then $\alpha(y,x)=\sv$. It is an $S$-tree  {\em over $\F\sub 2^{\vec S}$} if, in addition, for every node $t$ of~$T$ we have $\alpha(\vec F_t)\in\F$, where
 $$\vec F_t := \{(x,t) : xt\in E(T)\}.$$

\noindent
 We shall call the set $\vec F_t\sub\vec E(T)$ the {\em oriented star at~$t$} in~$T$ (even if it is empty).%
   \COMMENT{}
   Its image $\alpha(\vec F_t)\in\F$ is said to be {\em associated with\/} $t$ in $(T,\alpha)$.

An important example of ${S}$-trees are (irredundant)%
   \COMMENT{}
   ${S}$-trees \emph{over stars}: those over some $\F$ all of whose elements are stars of separations.%
   \COMB{`irredundant' put in brackets (Referee~1)\\[6pt] Footnote added (Referee~2)}%
   \footnote{For example, a~\td\ of width~$<w$ and adhesion~$<k$ of a graph is an $S_k$-tree for the set $S_k$ of separations of order~$<k$ over the set $\F_w$ of stars $\{(A_1,B_1),\dots,(A_n,B_n)\}$ such that $|\bigcap_{i=1}^n B_i| < w$. See Section~\ref{sec:twd}.}
   In such an ${S}$-tree $(T,\alpha)$ the map~$\alpha$ preserves the \emph{natural partial ordering} on $\vec E(T)$ defined by letting $(x,y) < (u,v)$ if $\{x,y\}\ne \{u,v\}$ and the unique $\{x,y\}$--$\{u,v\}$ path in $T$ joins $y$ to~$u$ (see Figure~\ref{Streeseps}).

   \begin{figure}[htpb]
\centering%\vskip-6pt
   	  \includegraphics{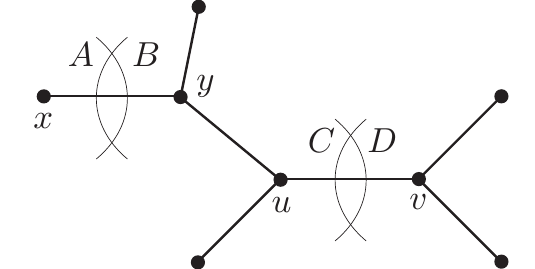}\label{Streeseps}
   	  \caption{Edges $(x,y) < (u,v)$ and separations $(A,B)= \alpha(x,y)\le\alpha(u,v)=(C,D)\!\!\!\!\!\!\!\!$}
   \label{fig:Stree}\vskip-6pt
   \end{figure}

\section{Tangle-tree duality in abstract separation\\ systems}\label{sec:AbstractDuality}

The tangle-tree duality theorem for abstract separation systems, the result from~\cite{TangleTreeAbstract} which we seek to apply in this paper to various different contexts, says the following. Let $(\vS,\le)$ be a separation system and $\F$ a collection of `forbidden' sets of separations. Then, under certain conditions, either $S$ has an $\F$-tangle or there exists an $S$-tree over~$\F$. We now define these conditions and state the theorem formally. We then prove a couple of lemmas that will help us apply it.

Let $\vr$ be a nontrivial and nondegenerate element of a separation system ${(\vS,\le\,,\!{}^*)}$ contained in some universe $(\vU,\le\,,\!{}^*,\vee,\wedge)$ of separations, the ordering and involution on~$\vS$ being induced by those of~$\vU$. Consider any $\vso\in\vS$ such that $\vr\le\vso$. As $\vr$ is nontrivial and nondegenerate, so is~$\vso$.

Let $S_{\ge\vr}$ be the set of all separations $s\in S$ that have an orientation ${\vs\ge\vr}$. Since $\vr$ is nontrivial, only one of the two orientations $\vs$ of every $s\in S_{\ge\vr}\sm\{r\}$ satisfies $\vs\ge\vr$. Letting
 $$f\shift(\!\vr)(\vso) (\vs) := \vs\vee\vso\quad\text{and}\quad f\shift(\!\vr)(\vso) (\sv) := (\vs\vee\vso)^*$$
 for all $\vs \ge \vr$ in $\vSr\sm\{\rv\}$ thus defines a map $\vSr\to\vU$, the \emph{shifting map}~$f\shift(\!\vr)(\vso)$ (Fig.~\ref{fig:shiftsep}, right). Note that $f\shift(\!\vr)(\vso)(\vr) = \vso$, since $\vr\le\vso$. Shifting maps preserve the partial ordering on a separation system, and in particular map stars to stars:

   \begin{figure}[htpb]
\centering
   	  \includegraphics{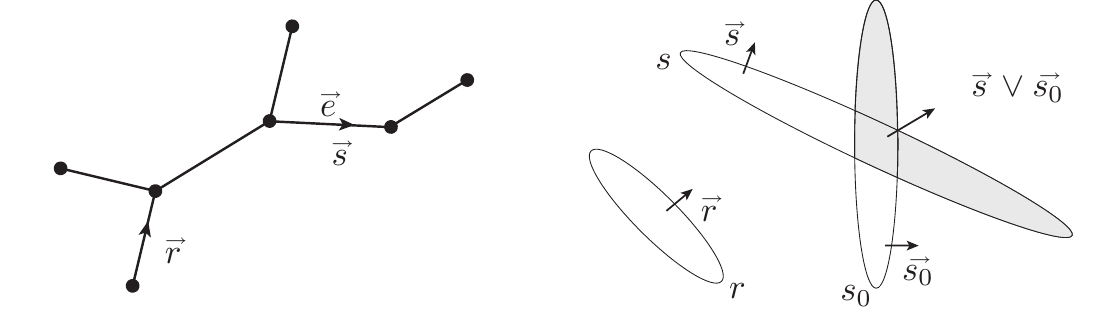}
   	  \caption{Shifting $\vs$ to $\vs\vee\vso$}
   \label{fig:shiftsep}\vskip-6pt
   \end{figure}

\begin{LEM}\label{lem:shifting} {\rm\cite{TangleTreeAbstract}}
The map $f = f\shift(\!\vr)(\vso)$ preserves the ordering $\le$ on $\vSr\sm\{\rv\}$.%
   \COMMENT{}
   In particular, $f$~maps stars to stars.%
   \COMMENT{}
\end{LEM}

Let us say that $\vso$ \emph{emulates}~$\vr$ in~$\vS$ if $\vso\ge\vr$ and every $\vs\in\vS\sm\{\rv\}$ with $\vs\ge\vr$ satisfies $\vs\vee\vso\in\vS$. We call $\vS$ \emph{separable} if for every two nontrivial and nondegenerate $\vr,\rvdash\in\vS$ such that $\vr\le\vrdash$ there exists an $s_0\in S$ with an orientation $\vso$ that emulates~$\vr$ and its inverse $\svo$ emulating~$\rvdash$.%
   \COMMENT{}

Given a set $\F\sub 2^{\vU}\!$ of stars of separations, we say that $\vso\in\vS$ \emph{emulates}~${\vr\in\vS}$ in~$\vS$ {\em for~$\F$} if $\vso$ emulates $\vr$ in~$\vS$ and for any star $\sigma\sub \vSr\sm\{\rv\}$ in~$\F$ that has an element $\vs\ge\vr$ we also have $f\shift(\!\vr)(\vso) (\sigma)\in\F$.%
   \COMMENT{}

Let us say that a set $\F$ {\it forces\/} the separations $\vs\in\vS$ for which $\{\sv\}\in\F$. And that $\vS$ is \emph{$\F$-separable} if for all nontrivial and nondegenerate $\vr,\rvdash\in\vS$ that are not forced by~$\F$ and satisfy $\vr\le\vrdash$ there exists an $s_0\in S$ with an orientation~$\vso$ that emulates $\vr$ in~$\vS$ for~$\F$ and such that $\svo$ emulates $\rvdash$ in~$\vS$ for~$\F$. (As earlier, any such $\vso$ will also be nontrivial and nondegenerate.)

Recall that an orientation $O$ of~$S$ is an {\em $\F$-tangle\/} if it is consistent and avoids~$\F$. We call $\F$ {\em standard\/} for~$\vS$ if it forces all $\vs\in\vS$ that are trivial in~$\vS$. The `strong duality theorem' from~\cite{TangleTreeAbstract} now reads as follows.

\begin{THM}[Tangle-tree duality theorem for abstract separation systems]\label{thm:strong} 
~\\  Let $(\vU,\le\,,\!{}^*,\vee,\wedge)$ be a universe of  separations containing a separation system $(\vS,\le\,,\!{}^*)$. Let $\F\sub 2^\vU\!\!$ be a set of stars, standard for~$\vS$. If $\vS$ is $\F$-separable,%
   \COMMENT{}
   exactly one of the following assertions holds:
\begin{enumerate}[\rm(i)]\itemsep0pt
  \item There exists an $\F$-tangle of~$S$.
  \item There exists an ${S}$-tree over $\F$.
  \end{enumerate}
\end{THM}

Often, the proof that $\vS$ is $\F$-separable can be split into two easier parts, a proof that $\vS$ is separable and one that $\F$ is \emph{closed under shifting} in~$\vS$: that whenever $\vso\in\vS$ emulates (in~$\vS$) some nontrivial and nondegenerate $\vr\le\vso$ not forced by~$\F$, then it does so for~$\F$. Indeed, the following lemma is immediate from the definitions:

\begin{LEM}\label{lem:Fsep}
  If $\vS$ is separable and $\F$ is closed under shifting in~$\vS$, then $\vS$ is $\F$-separable.\qed
\end{LEM}

The separability of~$\vS$ will often be established as follows.
Let us call a real function $\vs\mapsto\ord(\vs)$ on a universe $(\vU,\le\,,\!{}^*,\vee,\wedge)$%
   \COMMENT{}
   of oriented separations an \emph{order function} if it is non-negative, symmetric and submodular, that is, if $0\le \ord(\vs)=\ord(\sv)$ and 
\[\ord(\vr\vee\vs) + \ord(\vr\wedge\vs)\le \ord(\vr)+\ord(\vs)\]
for all $\vr,\vs\in \vU$. We then call $\ord(s) := \ord(\vs)$ the \emph{order} of $s$ and of~$\vs$. For every positive integer~$k$,
 $$\vS_k := \{\vs\in \vU : \ord(\vs) < k\}$$
is a submodular separation system (though not necessarily a universe).%
   \COMMENT{}

\begin{LEM}\label{lem:separable}
  Every such $\vS_k$ is separable.
\end{LEM}

\begin{proof}
Given nontrivial and nondegenerate $\vr,\rvdash\in\vS_k$ such that $\vr\le\vrdash$, we have to find an $\vso\in \vS_k$ such that $\vso$ emulates $\vr$ in~$\vS_k$ and $\svo$ emulates $\rvdash$ in~$\vS_k$. We choose $\vso\in\vU$ of minimum order with $\vr\le\vso\le\vrdash$.%
   \COMMENT{}
   Since $\vr$ is a candidate for~$\vso$, we have $\ord(\vso)\le\ord(\vr)$ and hence $\vso\in\vS_k$. We show that $\vso$ emulates~$\vr$; by symmetry, this will imply also that $\svo$ emulates~$\rvdash$.

Let us show that every $\vs\ge\vr$ in~$\vS_k$ satisfies $\vs\vee\vso\in\vS_k$.%
   \COMMENT{}
   We prove this by showing that $\ord(\vs\vee\vso)\le \ord(\vs)$,%
   \COMMENT{}
   which will follow from submodularity once we have shown that $\ord(\vs\wedge\vso)\ge \ord(\vso)$. This, however, holds since $\vs\wedge\vso$ was a candidate for the choice of~$\vso$: we have $\vr\le\vs\wedge\vso$ since $\vr\le\vs$ and $\vr\le\vso$, while $\vs\wedge\vso\le\vso\le\vrdash$.
\end{proof}

For the rest of this paper except in \ifArXiv Sections~\ref{sec:bwdsubmodular} and~\ref{sec:partitionsub}\else Section~\ref{sec:bwdsubmodular}\fi, whenever we consider a graph $G=(V,E)$ it will have at least one vertex, and we consider the universe $\vU$ of its (oriented) \emph{vertex separations}, the separations $(A,B)$ of~$V$ such that $G$ has no edge between $A\sm B$ and~$B\sm A$, with the order function
 $$\ord(A,B) := \abs{A\cap B}.$$
Note that $A$ and~$B$ are allowed to be empty.%
   \COMMENT{}
   For each positive integer~$k$, the set $\vS_k = \{\vs\in \vU : \ord(\vs) < k\}$ will be a submodular separable separation system, by Lemma~\ref{lem:separable}.

\section{Tangle-tree duality in graphs}\label{sec:bwd}

A \emph{tangle of order~$k$} in a finite%
   \COMMENT{}
    graph~$G = (V,E)$, as introduced by Robertson and Seymour~\cite{GMX}, is (easily seen to be equivalent to)%
   \COMMENT{}
          an orientation of~$S_k$ that avoids
 $$\T := \big\{\{(A_1,B_1), (A_2,B_2), (A_3,B_3)\}\sub \vU : G[A_1]\cup G[A_2]\cup G[A_3] = G\big\}.$$
(The three separations $(A_1,B_1), (A_2,B_2), (A_3,B_3)$ need not be distinct.) Clearly, $\T$~forces all the small separations in~$\vU$, those of the form~$(A,V)$.%
   \COMMENT{}
   Hence $\T\cap\vS_k$ is a standard subset of~$\vS_k$, for every integer~$k>0$.%
   \COMMENT{}

Notice that any $\T$-avoiding orientation $O$ of~$S_k$ is consistent, and therefore a $\T$-tangle in our sense, since for any pair of separations $(C,D)\le (A,B)$ we have $G[D]\cup G[A]\supe G[B]\cup G[A] = G$ and hence $\{(D,C),(A,B)\}\in\T$. Similarly, $O$~must contain all $(A,B)$ with $|A|<k$: it cannot contain $(B,A)$, as $(A,V)\in O$ by~$\{(V,A)\}\in\T$ but $\{(B,A),(A,V)\}\in\T$.

Since our duality theorems, so far, only work with sets $\F$ consisting of stars of separations, let us consider the set $\T^*$ of those sets in~$\T$ that are stars. 

\begin{THM}[Tangle-tree duality theorem for graphs]\label{TangleDuality}
~\\[1pt] For every $k>0$, every graph $G$%
   \COMMENT{}
   satisfies exactly one of the following assertions:
\vskip-6pt\vskip0pt
\begin{enumerate}[\rm(i)]\itemsep0pt
  \item $G$ has a $\T^*$-tangle of~$S_k$.
  \item $G$ has an ${S_k}$-tree over $\T^*$.
  \end{enumerate}
\end{THM}

\begin{proof}
By Theorem~\ref{thm:strong} and Lemmas~\ref{lem:Fsep}--\ref{lem:separable}, all we need to show is that $\T^*$ is closed under shifting in~$\vS = \vS_k$. This is easy from the definitions. Informally, if $(X,Y)\in\vS$ emulates some $\vr\le (X,Y)$ not forced by~$\T$%
   \COMMENT{}
   and we shift a star
  \[\sigma = \{(A_1,B_1), (A_2,B_2), (A_3,B_3)\}\sub \vSr\sm\{\rv\}\]%
 with $\vr\le (A_1,B_1)$, say, then we replace $(A_1,B_1)$ with~$(A_1\cup X,B_1\cap Y)$, and $(A_i,B_i)$ with $(A_i\cap Y,B_i\cup X)$ for $i\ge 2$. As any vertex or edge that is not in~$G[Y]$ lies in~$G[X]$, this means that ${\bigcup_i G[A_i] = G}$ remains unchanged.
\end{proof}

Our tangle-tree duality theorem can easily be extended to include the classical duality theorem of Robertson and Seymour~\cite{GMX} for tangles and branch-width. In order to do so, we first show that all $\T^*$-tangles are in fact $\T$-tangles, so these two notions coincide. Secondly, we will check that our tree structure witnesses for the non-existence of a tangle coincide with those used by Robertson and Seymour: that a graph has an $S_k$-tree over~$\T^*$ if and only if it has branch-width~${<k}$.

Using the submodularity of our order function $\{A,B\}\mapsto |A,B|$, we can easily show that $\T^*$-tangles of~$S_k$ are in fact $\T$-tangles:%
   \COMMENT{}

\begin{LEM}\label{T*toT}
Every consistent $\T^*$-avoiding orientation $O$ of~$S_k$ avoids~$\T$, as long as $|G|\ge k$.
\end{LEM}

\begin{proof}%
   \COMMENT{}
   Suppose $O$ has a subset $\sigma\in\T$. We show that as long as this set is not an inclusion-minimal nested set in~$\T$, we can either delete one of its elements, or replace it by a smaller separation in~$O$, so that the resulting set $\sigma'\sub O$ is still in~$\T$ but is smaller or contains fewer pairs of crossing separations. Iterating this process, we eventually arrive at a minimal nested set in~$\T$ that is still a subset of~$O$. By its minimality, this set is an antichain (compare the definition of~$\T$),%
   \COMMENT{}
   and all consistent nested antichains are stars.% 
   \footnote{Here we use that $|G|\ge k$: otherwise $\{(V,V)\}\in\T\sm\T^*$ lies in~$O$.}%
   \COMMENT{}
   Our subset of $O$ will thus lie in~$\T^*$, contradicting our assumption that $O$ avoids~$\T^*$.

If $\sigma$ has two comparable elements, we delete the smaller one and retain a subset of $O$ in~$\T$. We now assume that $\sigma$ is an antichain, but that it contains two crossing separations, $\vr = (A,B)$ and $\vs = (C,D)$ say. As these and their inverses lie in~$\vS_k$, submodularity implies that one of the separations $(A\cap D, {B\cup C})\le (A,B)$ and $(B\cap C, A\cup D)\le (C,D)$ also lies in~$\vS_k$. Let us assume the former; the other case is analogous.

Let $\sigma'$ be obtained from $\sigma$ by replacing $\vr$ with $\vrdash := (A\cap D, B\cup C)\in\vS_k$. Then $\sigma'$ is still in~$\T$, since any vertex or edge of $G[A]$ that is not in $G[A\cap D]$ lies in~$G[C]$, and $(C,D)$ is still in~$\sigma'$. Moreover, while $\vr$ crosses~$\vs$, clearly $\vrdash$ does not. To complete the proof, we just have to show that $\vrdash$ cannot cross any separation $\vt\in\sigma'$ that was nested with~$\vr$.

If $\vr\le\vt$ or $\vr\le\tv$, then $\vrdash\le\vr$ is nested with~$\vt$, as desired. If not then $\vt\le \vr$, since $\{\vr,\vt\}\sub O$ is consistent.%
   \COMMENT{}
   This contradicts our assumption that $\sigma$ is an antichain.
\end{proof}

The following elementary lemma provides the link between our $S$-trees and \bd{}s as defined by Robertson and Seymour~\cite{GMX}%
   \ifArXiv:\else. Its elementary proof is included in the ArXiv version of this paper~\cite{TangleTreeGraphsMatroidsArXiv}.\fi

\begin{LEM}\label{bwd-Stree}
For every integer $k \ge 3$,\/%
   \footnote{See the remark after Theorem~\ref{TangleThm}.}
   a graph $G$ of order at least~$k$%
   \COMMENT{}
   has branch-width~$<k$ if and only if $G$ has an $S_k$-tree over~$\T^*$.
\end{LEM}

\ifArXiv
\begin{proof}
  If $\abs{E(G)}\le 1$, we can obtain an $S_k$-tree over~$\T^*$ as follows. Let $\tau$ be any maximal nested set of 0-separations. Then $G$ has an $S_1$-tree $(T,\alpha)$%
   \COMMENT{}
   whose edges are $\alpha$-labelled by~$\tau$, with $\alpha(\vec F_t)\in\T^*$ for all internal nodes~$t$ of~$T$, and such that $\alpha$ maps every edge of~$T$ at a leaf, oriented towards that leaf, to a bipartition $(A,B)$ of~$V$ with either $|B|=\{v\}$ for some $v\in V$ or else $B=\{x,y\}$, where $xy$ is the unique edge of~$G$. We can extend this to an $S_3$-tree over~$\T^*$ by adding an edge labelled by~$\{V,\{v\}\}$ at every leaf of the first type, and by adding an edge labelled by $\{V,\{x,y\}\}$ at the unique leaf of the second type if it exists.

We now assume that $|E(G)|\ge 2$. Let us prove the forward implication first. We may assume that $G$ has no isolated vertices, because we can easily add a leaf in an $S_k$-tree corresponding to an isolated vertex.%
  \COMMENT{}

  Suppose $(T,L)$ is a \bd{} of width~$<k$. For each edge $e=st$ of $T$, let $T_s$ and $T_t$ be the components of $T-e$ containing $s$ and~$t$, respectively. Let $A_{s,t}$ and $B_{s,t}$ be the sets of vertices incident with an edge in $L^{-1}(V(T_s))$ and $L^{-1}(V(T_t))$, respectively,\vadjust{\penalty-200} except that if $A_{s,t}$ contains only two vertices we always put both these also in~$B_{s,t}$ (and similarly vice versa).%
   \COMMENT{}
   Note that $B_{s,t} = A_{t,s}$.

  For all adjacent nodes $s, t\in T$ let $\alpha(s,t) := (A_{s,t},B_{s,t})$. Since $G$ has no isolated vertices these $\alpha(s,t)$ are separations, and since $(T,L)$ has width~$<k$ they lie in~$\vS_k$. For each internal node $t$ of $T$ and its three neighbours $s_1,s_2,s_3$, every edge of $G$ has both ends in one of the~$G[A_{s_i,t}]$, so
 $${\alpha(\vec F_t) =\{\alpha(s_1,t),\alpha(s_2,t),\alpha(s_3,t)\}\in \T}.$$
 As $A_{s_i,t}\subseteq A_{t,s_j} = B_{s_j,t}$ for all $i\ne j$, the set $\alpha(\vec F_t)$ is a star. For leaves $t$ of~$T$, the associated star $\alpha(\vec F_t)$ has the form $\{(V,A)\}$ with $|A|=2$,%
   \COMMENT{}
   which is in $\T^*$ since $k\ge 3$. This proves that $(T,\alpha)$ is an $S_k$-tree over~$\T^*$.

  Now let us prove the converse. We may again assume that $G$ has no isolated vertex.%
  \COMMENT{}
  Let $(T,\alpha)$ be an $S_k$-tree over~$\T^*$. For each edge $e$ of $G$, let us orient the edges $st$ of $T$ towards $t$ whenever $\alpha(s,t)=(A,B)$ is such that $B$ contains both ends of~$e$. If $e$ has its ends in~$A\cap B$, we choose an arbitrary orientation of~$st$. As $T$ has fewer edges then nodes, there exists a node $t=:L(e)$ such that every edge at~$t$ is oriented towards~$t$.  

  Let us choose an $S_k$-tree $(T,\alpha)$ and $L\colon E\to V(T)$ so that the number of leaves in $L(E)$ is maximized, and subject to this with $\abs{V(T)}$ minimum. We claim that, for every edge $e$ of~$G$, the node $t=L(e)$ is a leaf of~$T$. Indeed, if not, let us extend~$T$ to make $L(e)$ a leaf. If $t$ has degree $2$, we attach a new leaf~$t'$ to~$t$ and put $\alpha(t',t)=(V(e),V)$ and $L(e)=t'$, where $V(e)$ denotes the set of ends of~$e$.
If $t$ has degree~$3$ then, by definition of~$\T$, there is a neighbour $t'$ of $t$ such that $e\in G[A]$ for $(A,B) = \alpha(t',t)$. As $t=L(e)$, this means that $e$ has both ends in~$A\cap B$. Subdivide the edge~$tt'$, attach a leaf $t^*$ to the subdividing vertex~$t''$, put $\alpha(t',t'')=\alpha(t'',t)=(A,B)$%
   \COMMENT{}
   and $\alpha(t^*,t'')=(V(e),V)$,%
   \COMMENT{}
   and let~$L(e)=t^*$. In both cases, $(T,\alpha)$ is still an $S_k$-tree over $\T^*$.%
   \COMMENT{}

In the same way one can show that $L$ is injective. Indeed, if ${L(e') = t = L(e'')}$ for distinct $e',e''\in E$, we could increase the number of leaves in~$L(E)$ by joining two new leaves $t',t''$ to the current leaf~$t$, letting $\alpha(t',t)=(V(e'),V)$ and $\alpha(t'',t)=(V(e''),V)$, and redefine $L(e')$ as~$t'$ and $L(e'')$ as~$t''$.

By the minimality of~$\abs{V(T)}$, every leaf of $T$ is in~$L(E)$, since we could otherwise delete it. Finally, no node $t$ of $T$ has degree~2, since contracting an edge at~$t$ while keeping $\alpha$ unchanged on the remaining edges would leave an $S_k$-tree over~$\T^*$. (Here we use that $G$ has no isolated vertices, and that $L(e)\ne t$ for every edge $e$ of~$G$.)%
   \COMMENT{}

Hence $L$~is a bijection from~$E$ to the set of leaves of~$T$, and $T$ is a ternary tree. Thus, $(T,L)$ is a \bd{} of~$G$, clearly of width less than~$k$.
\end{proof}
  \fi

We can now derive, and extend, the Robertson-Seymour~\cite{GMX} duality theorem for tangles and branch-width:

\begin{THM}[Tangle-tree duality theorem for graphs, extended]\label{TangleThm}
~\\[1pt] The following assertions are equivalent for all finite graphs $G\ne\emptyset$ and~${0<k\le |G|}\!:$
\vskip-6pt\vskip0pt
\begin{enumerate}[\rm(i)]\itemsep=0pt
\item $G$ has a tangle of order~$k$.
\item $G$ has a $\T$-tangle of $S_k$.
\item $G$ has a $\T^*$-tangle of $S_k$.
\item $G$ has no $S_k$-tree over $\T^*$.
\item $G$ has branch-width at least~$k$, or $k=1$ and $G$ has no edge,%
   \COMMENT{}
   or $k = 2$ and $G$ is a disjoint union of stars and isolated vertices and has at least one edge.\looseness=-1 %
   \COMMENT{}
\end{enumerate}
\end{THM}

\begin{proof}
   If $k=1$, then all statements are true. If $k=2$, they are all true if $G$ has an edge, and all false if not. Assume now that $k\ge 3$.

(i)$\leftrightarrow$(ii) follows from the definition of a tangle at the start of this section, and our observation that they are consistent.%
   \COMMENT{}

(ii)$\to$(iii) is trivial; the converse is Lemma~\ref{T*toT}.

(iii)$\leftrightarrow$(iv) is an application of Theorem~\ref{thm:strong}.

(iv)$\leftrightarrow$(v) is Lemma~\ref{bwd-Stree}.
\end{proof}

The exceptions in (v) for $k\le 2$ are due to a quirk in the notion of branch-width, which results from its emphasis on separating individual edges. The branch-width of all nontrivial trees other than stars is~2, but it is 1 for stars~$K_{1,n}$. For a clean duality theorem (even one just in the context of~\cite{GMX}) it should be~2 also for stars: every graph with at least one edge has a tangle of order~$2$, because we can orient all separations in~$S_2$ towards a fixed edge. Similarly, the branch-width of a disjoint union of edges is~0, but its tangle number is~2.%
   \COMMENT{}%

\section{Tangle-tree duality for set separations:\\ rank-width, carving-width and edge-tangles in graphs; matroid tangles; clusters in data sets}\label{sec:bwdsubmodular}

The concepts of branch-width and tangles were introduced by Robertson and Seymour~\cite{GMX} not only for graphs but more generally for hypergraphs. They proved all their relevant lemmas more generally for arbitrary order functions%
   \COMMENT{}
   $(A,B)\mapsto |A,B|$ rather than just $|A,B| = |A\cap B|$. Geelen, Gerards, Robertson, and Whittle~\cite{BranchDecMatroids} applied this explicitly to the submodular connectivity function in matroids.

Our first aim in this section is to derive from Theorem~\ref{thm:strong} a duality theorem for tangles in arbitrary universes of set separations%
   \footnote{Recall that these are more general than set partitions: the two sides may overlap.}
  equipped with an order function. This will imply the above branch-width duality theorems for hypergraphs and matroids,%
   \COMMENT{}
   as well as their cousins for carving width~\cite{ratcatcher}%
   \COMMENT{}
   and rank-width of graphs~\cite{rankwidth}.%
   \COMMENT{}
   It will also yield a duality theorem for {\em edge-tangles\/}, tangles of bipartitions of the vertex set of a graph whose order is the number of edges across. We shall then recast the theorem in the language of cluster analysis to derive a duality theorem for the existence of clusters in data sets.

Recall that an oriented {\em separation\/} of a set~$V$ is a pair $(A,B)$ such that $A\cup B = V$. Often, the separations considered will be bipartitions of~$V\!$, but in general we allow $A\cap B\ne\emptyset$. We also allow $A$ and~$B$ to be empty.%
   \COMMENT{}
   Recall that {\em order functions\/} are non-negative, symmetric and submodular functions on a separation system.

Let $\vU$ be any universe of separations of a set~$V\!$%
   \COMMENT{}
   of at least two elements,%
   \COMMENT{}
   with a submodular order function%
   \COMMENT{}
   $(A,B)\mapsto \ord(A,B)$. Given $k>0$, call an orientation of
 $$S_k = \{\,\{A,B\}\in U: \ord(A,B)<k\,\}$$
   a \emph{tangle of order~$k$} if it avoids

\medskip
\centerline{$\F = \big\{
  \{(A_1,B_1),(A_2,B_2),(A_3,B_3)\}\subseteq \vS_k:  A_1\cup A_2\cup A_3 = V
\,\big\}$}\nobreak\smallskip
\rightline{$\cup\ \{\,\{(A,B)\}\subseteq \vS_k : |B|=1\,\}$.}

\medskip\noindent
 Here, $(A_1,B_1), (A_2,B_2), (A_3,B_3)$ need not be distinct. In particular, $\F$~is standard%
   \COMMENT{}
   and tangles are consistent, so the tangles of~$\vU$ are precisely its $\F$-tangles.%
   \COMMENT{}

Let $\F^*\sub\F$ be the set of stars in~$\F$.%
   \COMMENT{}
   As in the proof of Theorem~\ref{TangleDuality}, it is easy to prove that $\F^*$ is closed under shifting in every~$\vS_k$. We  also have the following analogue of Lemma~\ref{T*toT}, with the same proof:

\begin{LEM}\label{lem:F*toFsub}
  Every consistent $\F^*$-avoiding orientation of~$S_k$ avoids~$\F$, as long as $|V|\ge k$.\qed
\end{LEM}

By Lemmas~\ref{lem:separable} and~\ref{lem:F*toFsub}, Theorem~\ref{thm:strong} now specializes as follows:

\begin{THM}[Tangle-tree duality theorem for set separations]\label{SetSeparations}
 ~\\[1pt] Given a universe $\vU$ of separations of a set~$V$ with a submodular order function, and~${k\le |V|}$, the following assertions are equivalent:
\vskip-3pt\vskip0pt
  \begin{enumerate}[\rm (i)]\itemsep=0pt
  \item $\vU$ has a tangle of order~$k$.%
     \COMMENT{}
  \item $\vU$ has an $\F^*$-tangle of~$S_k$.
  \item $\vU$ has no $\vS_k$-tree over~$\F^*$.\qed
  \end{enumerate}
\end{THM}

%\ifArXiv\goodbreak\fi

Applying Theorem~\ref{SetSeparations} with the appropriate order functions yields duality theorems for all known width parameters based on set separations. For example, let $V$ be the vertex set of a graph~$G$, with bipartitions as separations. Counting the edges across a bipartition defines an order function%
   \COMMENT{}
   whose $\F$-tangles are known as the {\em edge-tangles\/} of~$G$,%
   \COMMENT{}
   so Theorem~\ref{SetSeparations} yields a duality theorem for these. See Liu~\cite{LiuEdgeTangles} for more on edge-tangles, as well as their applications to immersion problems.%
   \COMMENT{}

The duals to edge-tangles of order~$k$ are $S_k$-trees over~$\F^*$. These were introduced by Seymour and Thomas~\cite{ratcatcher} as {\em carvings\/}. The least~$k$ such that $G$ has a carving is its {\em carving-width\/}. We thus have a duality theorem between edge-tangles and carving-width. 

Taking as the order of a vertex bipartition the rank of the adjacency matrix of the bipartite graph that this partition induces (which is submodular~\cite{rankwidth})%
   \COMB{source for submodularity cited (Referee~1)}
   gives rise to a width parameter called rank-width. In our terminology, $G$~has {\em rank-width\/}~$<k$ if and only if it admits an $S_k$-tree over~$\F^*$. The corresponding $\F$-tangles%
   \COMMENT{}
   of~$S_k$, then, are necessary and sufficient witnesses for having rank-width~$\ge k$, and we have a duality theorem for rank-width.

If $V$ is the vertex set of a hypergraph or the ground set of a matroid, the $\F$-tangles coincide, just as for graphs, with the hypergraph tangles of~\cite{GMX} or the matroid tangles of~\cite{BranchDecMatroids}. As in the proof of Lemma~\ref{bwd-Stree}, a hypergraph or matroid has branch-width $<k$ if and only if it has an $S_k$-tree over~$\F^*$. Theorem~\ref{SetSeparations} thus yields the original duality theorems of \cite{GMX} and~\cite{BranchDecMatroids} in this case.

\medbreak

  Our tangle-tree duality theorem for set separations can also be applied in contexts quite different from graphs and matroids. As soon as a set comes with a natural type of set%
   \COMMENT{}
   separation~-- for example, bipartitions~-- and a (submodular) order function on these, it is natural to think of the tangles in this separation  universe as clusters in that set. Theorem~\ref{SetSeparations} then applies to these clusters: if there is no cluster of some given order, then this is witnessed by a nested set of separations which cut the given set, recursively,%
   \COMMENT{}
   into small pieces.

The interpretation is that the separations to be oriented have small enough order that they cannot cut right through a cluster.%
   \COMMENT{}
   So if there exists a cluster, it can be thought of%
   \COMB{rephrased (Referee~1)}
   as orienting all these separations towards it. If not, the nested subset of the separations returned by the theorem divides the ground set into pieces too small to accommodate a cluster. This tree set of separations, therefore, will be an easily checkable witness for the non-existence of a cluster.

This approach to clusters has an important advantage over more traditional ways of identifying clusters: real-world clusters tend to be fuzzy, and tangles can capture them despite their fuzziness. For example, consider a large grid in a graph. For every low-order separation, most of the grid will lie on the same side, so the grid `orients' that separation towards this side. But every single vertex will lie on the `wrong' side for {\em some\/} low-order separation, the side not containing most of the grid; for example, it may be separated off by its four neighbours. The grid, therefore, defines a unique $k$-tangle for some large~$k$, but the `location' of this tangle is not represented correctly by any one of its vertices~-- just as for a fuzzy cluster in a data set it may be impossible to say which data exactly belong to that cluster and which do not.

Even if we base our cluster analysis just on bipartitions, we still need to define an order function to make this work. This will depend both on the type of data that our set represents and on the envisaged type of clustering. In~\cite{MonaLisa} there are some examples of how this might be done for a set of pixels of an image, where the clusters to be captured are the natural regions of this image such as a nose, or a cheek, in a portrait of the Mona Lisa. The corresponding duality theorem then reads as follows:

\begin{COR}\label{TDapplied} {\rm\cite{MonaLisa}}
For every picture $\pi$ on a canvas and every integer~$k>0$, either $\pi$ has a non-trivial region of coherence at least~$k$,%
   \COMMENT{}
   or there exists a laminar set of lines of order~$<k$ all whose splitting stars are void 3-stars%
   \COMMENT{}
    or single pixels. For no picture do both of these happen at once.
\end{COR}

\section{Tangle duality for tree-width in graphs}\label{sec:twd}

We now apply our abstract duality theorem to obtain a new duality theorem for tree-width in graphs. Its witnesses for large tree-width will be orientations of~$S_k$, like tangles, and thus different from brambles (or `screens'), the dual objects in the  classical tree-width duality theorem of Seymour and Thomas~\cite{ST1993GraphSearching}. 

This latter theorem, which ours easily implies, says that a finite graph either has a \td\ of width less than~$k-1$ or a bramble of order at least~$k$, but not both. The original proof of this theorem is as mysterious as the result is beautiful. The shortest known proof is given in~\cite{DiestelBook10noEE} (where we refer the reader also for definitions), but it is hardly less mysterious. A~more natural, if slightly longer, proof due to Mazoit is presented in~\cite{DiestelBook16}. The proof via our abstract duality theorem, as outlined below, is perhaps not shorter all told, but it seems to be the simplest available so far.

Given a finite graph $G = (V,E)$, we consider its separation universe~$\vU$ and the submodular separation systems~$\vS_k\sub\vU$ as defined at the end of Section~\ref{sec:AbstractDuality}.%
   \COMMENT{}
   For every integer $k>0$ let
 $$\textstyle \F_k := \big\{\, \sigma\sub \vU \mid \sigma = \{(A_i,B_i):i=0,\ldots,n\} \text{ is a star with } \big|\bigcap_{i=0}^n B_i\big|< k\,\big\}.$$
  (We take $\bigcap_{i=0}^n B_i := V$ if $\sigma=\emptyset$, so $K_1$ is an $S_k$-tree over~$\F_k$ if $|G|<k$.)%
   \COMMENT{}

Since $\F_k$ forces all the small nondegenerate separations in~$\vS_k$, the separations $(A,V)\in\vS_k$ with $A\ne V\!$,%
   \COMMENT{}
   it is standard for every~$S_k$. We have also seen that $\vS_k$ is separable (Lemma~\ref{lem:separable}). To apply Theorem~\ref{thm:strong} we thus only need the following lemma (cf.~Lemma~\ref{lem:Fsep})~-- whose proof contains the only bit of magic now left in tree-width duality:

\begin{LEM}\label{lem:twd-shiftable}
 For every integer $k>0$, the set $\F_k$ is closed under shifting in~$S_k$.
\end{LEM}

\begin{proof}
  Consider a separation $\vso=(X,Y)\in\vS_k =: \vS$ that emulates, in~$\vS$, some nontrivial and nondegenerate $\vr\in\vS$ not forced by~$\F_k$. Let
 $$\sigma = \big\{(A_i,B_i):i=0,\ldots,n\big\}\sub {\vSr}\sm\{\rv\}$$
 be a star in~$\F_k$ with $\vr\le (A_0,B_0)$.%
   \COMMENT{}
   Then
\begin{equation}\label{eq:S1}
   \vr\le (A_0,B_0)\le (B_i,A_i) \text{ for all }i\ge 1.
\end{equation}
  We have to show that
 $$\sigma' = \big\{(A'_i,B'_i): i=0,\dots,n \big\}\in \F_k$$
  for $(A'_i,B'_i) := f\shift(\!\vr)(\vso) (A_i,B_i)$.

%\ifArXiv\goodbreak\fi

From Lemma~\ref{lem:shifting} we know that $\sigma'$ is a star.%
   \COMMENT{}
   Since $(X,Y)$ emulates $\vr$ in~$\vS$,%
   \COMMENT{}
   we have $\sigma'\sub \vS$ by~\eqref{eq:S1}.%
   \COMMENT{}
   It remains to show that $\big|\bigcap_{i=0}^n B'_i\big| < k$. The trick will be to rewrite this intersection as the intersection of the two sides of a suitable separation that we know to lie in~$S=S_k$.%
   \COMMENT{}

By~\eqref{eq:S1} we have $(A'_0,B'_0) = (A_0\cup X,B_0\cap Y)$, while $(A'_i,B'_i) = (A_i\cap Y,B_i\cup X)$ for $i\ge 1$.
  Since the $(A_i,B_i)$ are separations, i.e.\ in~$\vU$, 
  so is $\big(\bigcap_{i=1}^n  B_i,\bigcup_{i=1}^n A_i\big)$.
  As trivially $(V,B_0)\in\vU$, this implies that, for $B^*:= \bigcap_{i=1}^n B_i$, also%
    \COMMENT{}
  $$\Big(\bigcap_{i=1}^n B_i \cap V\,,\, \bigcup_{i=1}^n A_i\cup B_0\Big) \specrel={\eqref{eq:S1}} (B^*,B_0) \in\,\vU\,.$$
Since $\sigma\in\F_k$ we have $\abs{B^*,B_0} = \big|\bigcap_{i=0}^n B_i\big| < k$, so $(B^*,B_0)\in \vS_k = \vS$ (Fig.~\ref{fig:shiftingTDs}).%
   \footnote{The~$A_i$, of course, are `more disjoint' than they appear in the figure.}
 As also $\vr\le (B^*,B_0)$ by~\eqref{eq:S1},\COMB{RHS of Figure~4 edited (Referee~3)}
   \begin{figure}[htpb]
     \centering
   	  \includegraphics{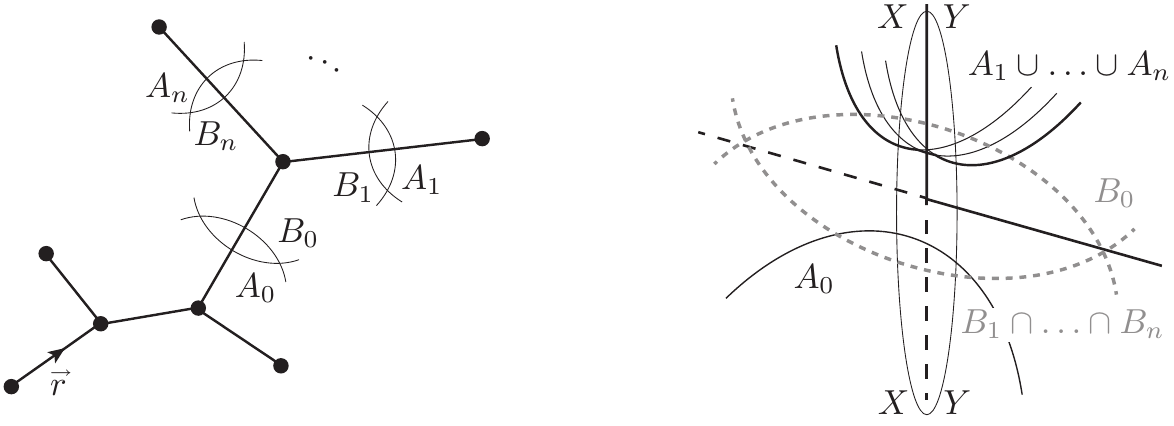}
   	  \caption{Shifting the separation $(B^*,B_0)$}
      \label{fig:shiftingTDs}
    \end{figure}%
   \COMMENT{}
   the fact that $(X,Y)$ emulates $\vr$ in~$\vS$%
   \COMMENT{}
  therefore implies that $(B^*\cup X, B_0\cap Y)\in \vS=\vS_k$. But then
 $$\Big|\bigcap_{i=0}^n B'_i\Big| = \Big|(B_0\cap Y) \cap \bigcap_{i=1}^n (B_i\cup X)\Big| = \big| B_0\cap Y, B^*\cup X\big| < k\,, $$
which means that $\sigma'\in\F_k$.
   \end{proof}

\begin{THM}[Tangle-treewidth duality theorem for graphs]\label{TangleTWDuality}
~\\[1pt] For every $k > 0$, every graph $G$%
   \COMMENT{}
   satisfies exactly one of the following assertions:
\vskip-6pt\vskip0pt
\begin{enumerate}[\rm(i)]\itemsep0pt
  \item $G$ has an $\F_k$-tangle of~$S_k$.
  \item $G$ has an $S_k$-tree over~$\F_k$.
  \end{enumerate}
\end{THM}

\begin{proof}
   Apply Theorem~\ref{thm:strong} and Lemmas~\ref{lem:Fsep}, \ref{lem:separable} and~\ref{lem:twd-shiftable}.
\end{proof}

\goodbreak

Condition~(ii) above can be expressed in terms of the tree-width ${\rm tw}(G)$ of~$G$:

\begin{LEM}\label{lem:twd-tree}
 A~graph $G$ has an $S_k$-tree over~$\F_k$ if and only if ${\rm tw}(G) < k-1$. More precisely, $G$~has an $S_k$-tree $(T,\alpha)$ over~$\F_k$ if and only if it admits a \td\ $(T,\V)$ of width~$< k-1$.%
   \COMMENT{}
 \end{LEM}

\begin{proof}
Given any $S$-tree $(T,\alpha)$ of~$G$ over a set $\F$ of stars, let $\V = (V_t)_{t\in T}$ be defined by letting
\begin{equation}\label{Vt}
 V_t = \bigcap\big\{\, B : (A,B) = \alpha(s,t),\ st\in E(T)\big\}.
\end{equation}
It is easy to check~\cite{confing} that $(T,\V)$ is a \td\ of~$G$ with adhesion sets $V_t\cap V_{t'} = A\cap B$ whenever $(A,B) = \alpha(t,t')$. If $S=S_k$ and $\F=\F_k$, we have $|V_t| < k$ at all $t\in T$,%
   \COMMENT{}
   so $(T,\V)$ has width less than~$k-1$.

Conversely, given a \td\ $(T,\V)$ with $\V = (V_t)_{t\in T}$, say, define $\alpha\colon\vec E(T)\to \vS_k$ as follows. Given $t_1 t_2\in E(T)$, let $T_i$ be the component of $T-t_1 t_2$ containing~$t_i$, and put $U_i := \bigcup_{t\in V(T_i)} V_t$ ($i=1,2$). Then let $\alpha(t_1, t_2):= (U_1, U_2)$. One easily checks~\cite{DiestelBook16} that $U_1\cap U_2 = V_{t_1}\cap V_{t_2}$, so $\alpha$ takes its values in~$\vS_k$ if $(T,\V)$ has width~$<k-1$. Moreover, every part $V_t$ satisfies~\eqref{Vt},%
   \COMMENT{}
   so if $(T,\V)$ has width $<k-1$ then $(T,\alpha)$ is over~$\F_k$.
\end{proof}

If desired, we can derive from Theorem~\ref{TangleTWDuality} the tree-width duality theorem of Seymour and Thomas~\cite{ST1993GraphSearching}. This is cast in terms of brambles, or `screens', as they originally called them. (See~\cite{DiestelBook16} for a definition and some background.)

Brambles have an interesting history. After Robertson and Seymour had invented tangles, they looked for a tangle-like type of highly cohesive substructure, or~HCS, dual to low tree-width. Their plan was that this should be a~map $\beta$ assigning to every set $X$ of fewer than $k$ vertices one component of~$G-X$. The question, in our language, was how to make these choices consistent: so that they would define an abstract~HCS.

The obvious consistency requirement, that $\beta(Y)\sub \beta(X)$ whenever $X\sub Y$, is easily seen to be too weak.%
   \COMMENT{}
   Yet asking that $\beta(X)\cap \beta(Y)\ne\emptyset$ for all~$X,Y$ turned out to be too strong. In~\cite{ST1993GraphSearching}, Seymour and Thomas then found a requirement that worked: that any two such sets, $\beta(X)$ and~$\beta(Y)$, should \emph{touch}: that either they share a vertex or $G$ has an edge between them. Such maps $\beta$ are now called \emph{havens}, and it is easy to show that $G$ admits a haven of \emph{order~$k$} (one defined on all sets $X$ of fewer than $k$ vertices) if and only if $G$ has a bramble of order at least~$k$.

\begin{LEM}\label{BrambleTranslation}
  $G$ has a bramble of order at least $k$ if and only if $G$ has an $\F_k$-tangle of~$S_k$.
\end{LEM}

\begin{proof}
  Let $\B$ be a bramble of order at least $k$. For every $\{A,B\}\in S_k$, since $A\cap B$ is too small to cover~$\B$ but every two sets in $\B$ touch and are connected, exactly one of the sets $A\sm B$ and $B\sm A$ contains an element of~$\B$. Thus,
  \[O=\{\, (A,B)\in \vS_k : B\sm A\text{ contains an element of } \B\,\}\]
  is an orientation of~$S_k$, which for the same reason is clearly consistent.%
   \COMMENT{}

  To show that $O$ avoids $\F_k$, let $\sigma = \{(A_1,B_1),\ldots,(A_n,B_n)\}\in \F_k$ be given. Then $\big|\bigcap_{i=1}^n B_i\big|<k$, so some $C\in\B$ avoids this set and hence lies in the union of the sets $A_i\sm B_i$. But these sets are disjoint, since $\sigma$ is a star, and have no edges between them.%
   \COMB{`no edges' added, thanks to Referee~1!}
   Hence $C$ lies in one of them, $A_1\sm B_1$ say, putting $(B_1,A_1)$ in~$O$. But then $(A_1,B_1)\notin O$, so $\sigma\not\sub O$ as claimed.

  Conversely, let $O$ be an $\F_k$-tangle of~$S_k$. We shall define a bramble~$\B$ containing for every set $X$ of fewer than $k$ vertices exactly one component of~$G-X$, and no other sets. Such a bramble will have order at least~$k$, since no such set $X$ covers it.

Given such a set~$X$, note first that $X\ne V$. For if $|V|<k$ then $\emptyset\in\F_k$, contradicting our assumption that $O$ has no subset in~$\F_k$. Let $C_1,\ldots,C_n$ be the vertex sets of the components of~$G-X$. Consider the separations $(A_i,B_i)$ with $A_i = C_i\cup N(C_i)$ and $B_i = V\sm C_i$. Since
 $$\sigma_X := \{\,(A_i,B_i) \mid\, i=1,\dots,n\,\}$$
 is a star in~$\F_k$, not all the $(A_i,B_i)$ lie in~$O$. So $(B_i,A_i)\in O$ for some~$i$, and since $O$ is consistent this $i$ is unique. Let us make $C_i$ an element of~$\B$.

   \begin{figure}[htpb]
\centering
   	  \includegraphics{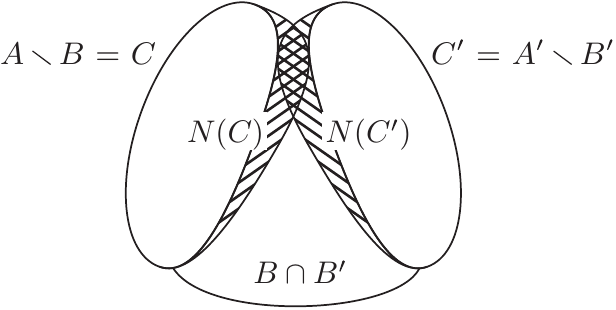}
   	  \caption{If $C,C'$ do not touch, then $(B,A)$ and $(B',A')$ are inconsistent.}
   \label{fig:touch}\vskip-3pt
   \end{figure}

It remains to show that every two sets in~$\B$ touch. Given $C,C'\in\B$, there are sets $X$ and $X'$ such that $\sigma_X$ contains a separation $(A,B)$ with $A = C\cup N(C)$ and $(B,A)\in O$, and likewise for~$C'$. If $C$ and~$C'$ do not touch, then $C'\sub B\sm A$ and hence $A'\sub B$ (Fig.~\ref{fig:touch}), and similarly $A\sub B'$. Hence $(A,B)\le (B',A')\in O$ but also $(B,A)\in O$, contradicting the consistency of~$O$.
   \end{proof}

Theorem~\ref{TangleTWDuality} can thus be extended to incorporate tree-width and brambles:

\begin{THM}[Tangle-bramble-treewidth duality theorem for graphs]\label{BrambleThm}\ifArXiv\else~\\[-9pt]\fi
The following assertions are equivalent for all finite graphs $G$%
   \COMMENT{}
  and~${k>0}\!:$
 \vskip-6pt\vskip0pt
\begin{enumerate}[\rm(i)]\itemsep=0pt
\item $G$ has a bramble of order at least~$k$.
\item $G$ has an $\F_k$-tangle of~$S_k$.
\item $G$ has no $S_k$-tree over~$\F_k$.
\item $G$ has tree-width at least~$k-1$.
\end{enumerate}
\end{THM}

\begin{proof}
(i)$\leftrightarrow$(ii) is Lemma~\ref{BrambleTranslation}.
(ii)$\leftrightarrow$(iii) is Theorem~\ref{TangleTWDuality}.
(iii)$\leftrightarrow$(iv) is Lemma~\ref{lem:twd-tree}.
\end{proof}

\goodbreak

\section{Tangle duality for path-width in graphs}\label{sec:pathwidth}

A \emph{\pd} of a graph $G$ is a \td\ of $G$ whose decomposition tree is a path. The \emph{path-width} of~$G$ is the least width of such a \td. By Lemma~\ref{lem:twd-tree}, $G$~has path-width~$<k-1$ if and only if it has an $S_k$-tree over%
   \COMB{Def of $\F_k^{(2)}$ rewritten (Referees 1 and 2)}
 $$\textstyle \F_k^{(2)} := \big\{\, \sigma\in~\F_k: |\sigma|\le 2\,\big\},$$
   where $\F_k$ is defined as in Section~\ref{sec:twd}.%
   \COMMENT{}
   Theorem~\ref{TangleTWDuality} has the following analogue:

\begin{THM}[Tangle-pathwidth duality theorem for graphs]\label{TanglePWDuality}
~\\[1pt] For every $k > 0$, every graph $G$ satisfies exactly one of the following assertions:
\vskip-6pt\vskip0pt
\begin{enumerate}[\rm(i)]\itemsep0pt
  \item $G$ has an $\F_k^{(2)}$-tangle of~$S_k$.
  \item $G$ has an $S_k$-tree over~$\F_k^{(2)}$.
  \end{enumerate}
\end{THM}

\begin{proof}
   Apply Theorem~\ref{thm:strong} and Lemmas~\ref{lem:Fsep}, \ref{lem:separable} and~\ref{lem:twd-shiftable}.
\end{proof}

Bienstock, Robertson, Seymour and Thomas~\cite{QuicklyExcludingForest} also found tangle-like HCSs dual to path-width, which they call `blockages'.%
   \footnote{They go on to show that all graphs with a blockage of order~$k-1$~-- which are precisely the graphs of path-width at least~$k-1$~-- contain every forest of order~$k$ as a minor. This corollary is perhaps better known than the path-width duality theorem itself.}
  Let us define these, and then incorprorate their result into our duality theorem with a unified proof.

Given a set $X$ of vertices in $G=(V,E)$, let us write $\partial(X)$ for the set of vertices in $X$ that have a neighbour outside~$X$. A \emph{blockage of order $k-1$}, according to~\cite{QuicklyExcludingForest}, is a collection $\B$ of sets $X\sub V$ such that
\begin{enumerate}[(B1)]\itemsep=0pt
\item $\abs{\partial(X)} < k$ for all $X\in\B$;
\item $X'\in\B$ whenever $X'\subseteq X\in\B$ and $\abs{\partial(X')} < k$;
\item for every $\{X_1,X_2\}\in S_k$, exactly one of $X_1$ and $X_2$ lies in~$\B$.
\end{enumerate}

To deduce the duality theorem of~\cite{QuicklyExcludingForest} from Theorem~\ref{thm:strong}, we just need to translate blockages into orientations of~$S_k$:%
   \COMMENT{}

\begin{THM}[Tangle-blockage-pathwidth duality theorem for graphs]\ifArXiv\else~\\[-9pt]\fi
The following assertions are equivalent for finite graphs $G\ne\emptyset$ and~${k>0}\!:$
 \vskip-6pt\vskip0pt
  \begin{enumerate}[\rm(i)]\itemsep=0pt
  \item $G$ has a blockage of order $k-1$.
  \item $G$ has an $\F_k^{(2)}$-tangle of~$S_k$.
  \item $G$ has no  $S_k$-tree over $\F_k^{(2)}$.
  \item $G$ has path-width at least $k-1$.
  \end{enumerate}
\end{THM}

\begin{proof}
Theorem~\ref{TanglePWDuality} asserts the equivalence of (ii) and~(iii), while (iii) is equivalent to~(iv) by Lemma~\ref{lem:twd-tree}.%
   \COMMENT{}

(i)$\to$(ii): Suppose that $G$ has a blockage~$\B$ of order $k-1$. By (B2) and (B3), 
  \[ O=\{\,(X,Y)\in \vS_k: X\in \B\,\}\] 
  is a consistent orientation of~$S_k$.

For a proof that $O$ avoids every singleton star $\{(A,X)\}\in\F_k^{(2)}$ it suffices to show that $\B$ contains every set~$X$ of order~$<k$: then $(X,A)\in O$ and hence $(A,X)\notin O$. To show that $X\in\B$, consider the separation $\{X,V\}\in S_k$. If $V\in\B$, then also $X\in\B$ by~(B2), contradicting~(B3). Hence $V\notin\B$, and thus $X\in\B$ by~(B3).

To complete the proof that $O$ avoids~$\F_k^{(2)}$ consider ${\{(A_1,B_1),(A_2,B_2)\}\in \F_k^{(2)}}$ with $(A_1,B_1)\neq (A_2,B_2)$, and suppose that $(A_1,B_1)\in O$. Since $\{(A_1,B_1),(A_2,B_2)\}$ is a star,%
   \COMMENT{}
   $\{B_1,B_2\}$ is a separation. As $|B_1\cap B_2|<k$ by definition of~$\F_k^{(2)}$, it lies in~$S_k$. Applying (B3) three times, we deduce from our assumption of $A_1\in \B$ that $B_1\notin \B$, and hence $B_2\in\B$, and hence $A_2\notin \B$. Thus, $(A_2,B_2)\notin O$.

(ii)$\to$(i): Let $O$ be an $\F_k^{(2)}$-tangle of~${S}_{k}$. We claim that
 $$\B := \{\, X : (X,Y)\in O\,\}$$
is a blockage of order $k-1$. Clearly, $\B$ satisfies (B1). 

Given $\{X_1,X_2\}\in S_k$ as in~(B3),%
   \COMMENT{}
   assume that $(X_1,X_2)\in O$. Then $X_1\in \B$. If also $X_2\in\B$, there exists $Y_2$ such that $(X_2,Y_2)\in O$. Then $\{X_1\cap Y_2,X_2\}$ is still a separation of~$V\!$,%
   \COMMENT{}
   and clearly in~$S_k$. As $(X_1\cap Y_2,X_2)\le (X_1,X_2)\in O$ and $O$ is consistent, we have $(X_1\cap Y_2,X_2)\in O$. Then $\{(X_1\cap Y_2,X_2),(X_2,Y_2)\}$ is a star in $\F_k^{(2)}$, contradicting our assumption. 

Given $X'\sub X\in\B$ as in~(B2), with $(X,Y)\in O$ say,%
   \COMMENT{}
   let $Y':= \partial(X')\cup (V\sm X')$ and $Z := \partial(X)\cup (V\sm X)$. Then $Z\sub Y$ and hence $|X\cap Z|\le |X\cap Y| < k$, so $\{X,Z\}\in S_k$. By (B3)%
   \COMMENT{}
   we have $Z\notin \B$ and hence $(Z,X)\notin O$, so $(X,Z)\in O$.%
   \COMMENT{}
   Since $O$ is consistent and $\vS_k\owns (X',Y')\le (X,Z)$,%
   \COMMENT{}
   we thus obtain $(X',Y')\in O$ and hence $X'\in\B$, as desired.
\end{proof}

\section{Tangle duality for tree-width in matroids}\label{sec:matroid}
Hlin{\v e}n{\'y} and Whittle~\cite{HlinenyWhittle,HlinenyWhittleErrata} generalized the notion of tree-width from graphs to matroids.%
   \footnote{In our matroid terminology we follow Oxley~\cite{OxleyBook}.}
   Let us show how Theorem~\ref{thm:strong} implies a duality theorem for tree-width in matroids.%
   \COMMENT{}

Let $M=(E,I)$ be a matroid with rank function $r$. Its \emph{connectivity function} is defined as
 $$\lambda(X) := r(X)+r(E\sm X)-r(M).$$
 We consider the universe $\vU$ of all bipartitions $(X,Y)$ of~$E$. Since
 $$\ord(X,Y):= \lambda(X) = \lambda(Y)$$
 is non-negative, submodular and symmetric, it is an order function on~$\vU$, so our universe $\vU$ is submodular.

A \emph{\td} of~$M$ is a pair~$(T,\tau)$, where $T$ is a tree and ${\tau\colon E\to V(T)}$ is any map.
Let $t$ be a node of~$T$, and let $T_1,\ldots,T_d$ be the components of $T-t$. Then the \emph{width} of~$t$ is the number%
   \COMMENT{}
\[ \sum_{i=1}^d r(E\sm F_i) - (d-1)\, r(M),\] 
where $F_i=\tau^{-1}(V(T_i))$. (If $t$ is the only node of~$T$, we let its width be~$r(M)$.) The \emph{width} of $(T,\tau)$ is the maximum width of the nodes of~$T$. The \emph{tree-width} of~$M$ is the minimum width over all \td{}s of~$M$.

\goodbreak

Matroid tree-width was designed so as to generalize the tree-width of graphs:

\begin{THM}[Hlin{\v e}n{\'y} and Whittle~\cite{HlinenyWhittle,HlinenyWhittleErrata}]
  The tree-width of a finite graph containing at least one edge equals the tree-width of its cycle matroid.
\end{THM}

In order to specialize Theorem~\ref{thm:strong} to a duality theorem for tree-width in matroids, we consider for $k>0$ the set
 $$S_k = \{\,\{A,B\}\in U: \ord(A,B)<k\,\};$$
 then $\vS_k$ is separable by Lemma~\ref{lem:separable}. For stars $\sigma = \{(A_i,B_i) :i = 0,\ldots,n\}\sub\vU$ we write 
 $$\starorder\sigma := r(M) + \sum_{i=0}^n \big( r(B_i) - r(M)\big)\,.$$%
   \COMMENT{}
   %\ifArXiv\else\vskip-9pt\noindent\fi
   We consider
 $$\F_k := \big\{\, \sigma\sub \vU : \sigma \text{ is a star with } \starorder\sigma <k\,\big\}.$$%
   \COMMENT{}

Clearly, the singleton stars $\{(A,B)\}$ in~$\F_k$ are precisely those with $r(B) < k$, and the empty star lies in~$\F_k$ if and only if $r(M)<k$. We remark that requiring $\sigma\sub \vS_k$ in the definition of~$\F_k$ would not spare us a proof of the following lemma, which we shall need in the proof of Lemma~\ref{lem:matroidTDs}.

\begin{LEM}\label{lem:matroidTW}
Every $\sigma\in\F_k$ is a subset of~$\vS_k$.
\end{LEM}
 
\begin{proof}
We show that every $A_i$ in a star $\sigma = \{(A_i,B_i):i=0,\ldots,n\}\sub\vU$ satisfies $\lambda(A_i)\le\starorder\sigma$; if $\sigma\in\F_k$, this implies that $\ord(A_i,B_i)<k$ as desired. Our proof will be for $i=0$; the other cases then follow by symmetry. 
 
Since $\sigma$ is a star we have $A_i\sub B_j$ whenever $i\ne j$, and in particular $A_{i+1}\sub B_i^*:= B_1\cap\ldots\cap B_i$ for $i=1,\dots,n-1$.%
   \COMMENT{}
   Hence $B_i^*\cup B_{i+1}\supe E$. Submodularity of the rank function now gives%
   \COMMENT{}
\[
r(B_i^*)+r(B_{i+1})\ge r(B_i^*\cap B_{i+1}) + r(B_i^*\cup B_{i+1}) = r(B_{i+1}^*) + r(M)
\]%
   \COMMENT{}
   for each $i=1,\ldots,n-1$. Summing these inequalities over $i=1,\ldots,n-1$, and noting that $B^*_1 = B_1$, yields
 $$r(B_1)+\ldots+r(B_n)\ge r(B_1\cap \ldots\cap B_n) + (n-1)\, r(M).$$
Using that $\sigma$ is a star and hence $A_0 \sub B_1\cap\ldots\cap B_n$, we deduce
   %\ifArXiv\else\vskip-12pt\fi
\begin{align*} \starorder\sigma =
\sum_{i=0}^n r(B_i) - n\, r(M) &\ge r(B_0)+r(B_1\cap \ldots\cap B_n)-r(M)\\
 \noalign{\vskip-6pt}
   & \ge r(B_0)+r(A_0)-r(M)\\
 \noalign{\vskip3pt}
   & = \lambda(A_0).
\end{align*}%\ifArXiv\else\vskip-3pt\fi
as desired.
\end{proof}

In order to apply Theorem~\ref{thm:strong}, we have to prove that $\vS_k$ is $\F_k$-separable:

\begin{LEM}
  $\vS_k$ is $\F_k$-separable.
\end{LEM}

\begin{proof}
Let $\vr,\rvdash\in\vS_k$ be given: nondegenerate, nontrivial, not forced by~$\F_k$, and satisfying $\vr\le\vrdash$. Pick $(X,Y)\in\vS_k$ with $\vr\le (X,Y)\le\vrdash$ and $|X,Y|$ minimum.%
   \COMMENT{}
   We claim that $(X,Y)$ emulates $\vr$ in~$\vS_k$ for~$\F_k$, and that $(Y,X)$ emulates $\rvdash$ in~$\vS_k$ for~$\F_k$. By symmetry,%
   \COMMENT{}
   it is enough to prove that $(X,Y)$ emulates $\vr$ for~$\F_k$.%
   \COMMENT{}
   
The proof of Lemma~\ref{lem:separable} shows that $(X,Y)$ emulates~$\vr$.%
   \footnote{Technically, we do not need this fact at this point and could use Lemma~\ref{lem:matroidTW} to deduce it from the fact that all $\sigma'$ as below lie in~$\F_k$. But that seems heavy-handed.}%
   \COMMENT{}
   To show that it does so for~$\F_k$, consider a nonempty%
   \COMMENT{}
   star
 $$\sigma=\big\{(A_i,B_i):i=0,\ldots,n\big\}\sub \vSr\sm\{\rv\}$$
 in~$\F_k$ (where $\vS:=\vS_k$) with $\vr\le (A_0,B_0)$. Then
\begin{equation}\label{eq:S1matroid}
   \vr\le (A_0,B_0)\le (B_i,A_i) \text{ for all }i\ge 1.
\end{equation}
  We have to show that
 $$\sigma' = \big\{(A'_i,B'_i): i=0,\dots,n \big\}\in \F_k$$
  for $(A'_i,B'_i) := f\shift(\!\vr)(\vso) (A_i,B_i)$.

  From Lemma~\ref{lem:shifting} we know that $\sigma'$ is a star. Since $(X,Y)$ emulates~$\vr$,%
   \COMMENT{}
   we have $\sigma'\sub \vS_k$ by~\eqref{eq:S1matroid}.%
   \COMMENT{}
   It remains to show that $\starorder{\sigma'} < k$. We show that, in fact,
 \begin{equation}\label{eq:ranksum}
  \big(\starorder{\sigma'}=\big)\quad  r(Y\cap B_0)+\sum_{i=1}^n r(X\cup B_i)- n\, r(M) \le\starorder\sigma\,;
 \end{equation}
 as $\starorder\sigma < k$ by our assumption that $\sigma\in\F_k$, this will complete the proof.

By submodulary of the rank function, we have
  \begin{align}\nonumber
    r(Y\cap B_0)+r(Y\cup B_0)&\le r(Y)+r(B_0)\\
    \nonumber
    \text{and}\quad r(X\cup B_i)+r(X\cap B_i)&\le r(X)+r(B_i)\quad \text{for }i = 1,\dots,n.
  \end{align}
 For our proof of~\eqref{eq:ranksum} we need to show that the sum of the first terms in these $n+1$ inequalities is at most the sum of the last terms.%
   \COMMENT{}
   This will follow from these inequalities once we know that the sum of the second terms is at least the sum of the third terms. So let us prove this, i.e., that
 \begin{equation}\label{eq:ranknew}
 r(Y\cup B_0) + \sum_{i=1}^n r(X\cap B_i)\ge r(Y) + n\, r(X)\,.
 \end{equation}
 For $i=1,\dots,n$ let us abbreviate $A_i^* := A_1\cup\ldots\cup A_i$ and $B_i^* := B_1\cap\ldots\cap B_i$.

Since $\sigma$ is a star we have $A_i\sub B_j$ whenever $i\ne j$. Hence $A_n^*\subseteq B_0$, giving 
  \begin{equation}
   r(Y\cup B_0)\ge r(Y\cup A_n^*),
    \label{eq:matroid1}
  \end{equation}
and $A_{i+1}\sub B_i^*$ for $i \ge 1$. Hence $B_i^*\cup B_{i+1}\supe E$. By submodularity, this implies
  \begin{align}\nonumber
r(X\cap B_i^*)+r(X\cap B_{i+1})&\ge r(X\cap (B^*_i\cap B_{i+1})) + r(X\cap (B_i^*\cup B_{i+1}))\\
    &= r(X\cap B_{i+1}^*) + r(X)\nonumber
  \end{align}
   for each $i=1,\ldots,n-1$.\vadjust{\penalty-200} Summing this for $i=1,\ldots,n-1$, and recalling that $B^*_1 = B_1$, we obtain%
   \COMMENT{}
 \begin{equation}\label{eq:matroid2}
 \sum_{i=1}^n r(X\cap B_i)\ge r(X\cap B_n^*) + (n-1)\, r(X)\,.
 \end{equation}

Since $\{X,Y\}$ and $\{B_n^*,A_n^*\}$ are bipartitions of~$E$, so is $\{X\cap B_n^*,Y\cup A_n^*\}$. Moreover, we have $\vr\le (X\cap B_n^*,Y\cup A_n^*)$ since $\vr\le (X,Y)$ and $\vr\le (B^*_n,A^*_n)$ by~\eqref{eq:S1matroid}, and we also have $(X\cap B_n^*,Y\cup A_n^*)\le (X,Y)\le\vrdash$. It would therefore contradict our choice of~$(X,Y)$ if we had $\ord(X\cap B_n^*,Y\cup A_n^*) < \ord(X,Y)$.%
   \COMMENT{}
  Hence $\ord(X\cap B_n^*,Y\cup A_n^*) \ge \ord(X,Y)$, and therefore%
    \COMMENT{}
  \begin{equation}
  r(X\cap B_n^*) + r(Y\cup A_n^*) \ge r(X)+r(Y).
  \label{eq:matroid3}
  \end{equation}
 Adding up inequalities~\eqref{eq:matroid1}, \eqref{eq:matroid2},~\eqref{eq:matroid3} we obtain~\eqref{eq:ranknew}, proving~\eqref{eq:ranksum}.
\end{proof}

\begin{LEM}\label{lem:matroidTDs}
  $M$ has an ${S}_k$-tree over $\F_k$ if and only if  it has tree-width~$<k$.  More precisely, $M$~has an $S_k$-tree $(T,\alpha)$ over~$\F_k$ if and only if it admits a \td\ $(T,\tau)$ of width~$< k$.%
   \COMMENT{}
\end{LEM}

\begin{proof}
 For the forward implication, consider any $S_k$-tree $(T,\alpha)$ of~$M$. Given $e\in E$, orient every edge $st$ of~$T$, with $\alpha(s,t) = (A,B)$ say, towards~$t$ if $e\in B$, and let $\tau$ map $e$ to the unique sink of~$T$ in this orientation. Then $(T,\tau)$ is a \td\ of~$M$. If $(T,\alpha)$ is over $\F_k$, the decomposition is easily seen to have width less than~$k$.%
   \COMMENT{} 

Conversely, let $(T,\tau)$ be a \td\ of~$M$ of width~$<k$. For every edge $e=st$ of~$T$, let $T_s$ and~$T_t$ be the components of $T-e$ containing $s$ and~$t$, respectively. Let
 $$\alpha(s,t) := \big(\tau^{-1}(T_s),\tau^{-1}(T_t)\big)\in\vU.$$
   Since every node $t$ has width less than~$k$, its associated star $\{\,\alpha(s,t) : st\in E(T)\}$ of separations is in~$\F_k$. (This includes the case of~$|T|=1$.)%
   \COMMENT{}
   By Lemma~\ref{lem:matroidTW} this implies that $\alpha(\vec E(T))\sub \vS_k$, so $(T,\alpha)$~is an $S_k$-tree over~$\F_k$.\looseness=-1
\end{proof}

Theorem~\ref{thm:strong} now yields the following duality theorem for matroid tree-width.

\begin{THM}[Tangle-treewidth duality theorem for matroids]
~\\[1pt]
  Let $M$ be a matroid, and let $k>0$ be an integer. Then the following statements are equivalent:
  \begin{enumerate}[\rm (i)]\itemsep=0pt\vskip-3pt\vskip0pt
  \item $M$ has tree-width at least $k$.
  \item $M$ has no $S_k$-tree over $\F_k$.
  \item $M$ has an $\F_k$-tangle of~$S_k$.
 \qed \end{enumerate}
\end{THM}

\section{Tangle duality for tree-decompositions of small adhesion}\label{sec:adhesion}

To demonstrate the versatility of Theorem~\ref{thm:strong}, we now deduce a duality theorem for a new width parameter: one that bounds the width and the adhesion of a \td\ independently, that is, allows the first bound to be greater.

Recall that the \emph{adhesion} of a \td{} $(T,\V)$ of a  graph $G=(V,E)$ is the largest size of an attachment set, the number $\max_{st\in E(T)} \abs{V_{s}\cap V_t}$. (If $T$ has only one node~$t$, we set the adhesion to~0.) Trivially if a \td{} has width~$<k-1$ it has adhesion~$<k$, and it is easy to convert it to a \td\ of the same width and adhesion~$<k-1$.

The idea now is to have a duality theorem whose tree structures are the \td s of adhesion~$<k$ and width less than $w-1 \ge k-1$. For $w=k$ this should default to the duality for tree-width discussed in Section~\ref{sec:twd}. 

Let $\vU$ and $\vS_k$ be as defined at the end of Section~\ref{sec:AbstractDuality}. Recall that $\vS_k$ is separable, by Lemma~\ref{lem:separable}. Let 
\[ \F^w_k=\big\{\sigma\subseteq \vS_k\mid \sigma=\{(A_i,B_i):i=0,\ldots,n\}\text{ is a star with $\textstyle\big\lvert\bigcap_{i=0}^n B_i\big\rvert<w$}\big\},\]
(As before, we let $\bigcap_{i=0}^n B_i := V$ if $\sigma=\emptyset$, so $K_1$ is an $S_k$-tree over~$\F_k$ if $|G|<w$.) Note that, for $w=k$, we have $\F_k^w = \F_k$ as defined in Section~\ref{sec:twd}.

%\ifArXiv\goodbreak\fi

\begin{LEM}\label{lem:twdadh-shitable}
  $S_k$ is $\F^w_k$-separable.
\end{LEM}

\begin{proof}
  Let $\vr,\rvdash\in\vS_k$ be given: nondegenerate, nontrivial, not forced by $\F^w_k$, and satisfying $\vr\le\vrdash$. Pick $\vso = (X,Y)\in\vS_k$ with $\vr\le (X,Y)\le\vrdash$ and $|X,Y|$ minimum.%
   \COMMENT{}
   We claim that $(X,Y)$ emulates $\vr$ in~$\vS_k$ for~$\F^w_k$, and that $(Y,X)$ emulates $\rvdash$ in~$\vS_k$ for~$\F^w_k$. By symmetry,%
   \COMMENT{}
   it is enough to prove that $(X,Y)$ emulates $\vr$ for~$\F^w_k$. 

The proof of Lemma~\ref{lem:separable} shows that $(X,Y)$ emulates~$\vr$. To show that it does so in~$\F^w_k$, consider a nonempty%
   \COMMENT{}
   star
 $$\sigma=\big\{(A_i,B_i):i=0,\ldots,n\big\}\sub \vSr\sm\{\rv\}$$
 in~$\F^w_k$ (where $\vS:=\vS_k$) with $\vr\le (A_0,B_0)$. Then
\begin{equation}\label{eq:S2}
   \vr\le (A_0,B_0)\le (B_i,A_i) \text{ for all }i\ge 1.
\end{equation}
  We have to show that
 $$\sigma' = \big\{(A'_i,B'_i): i=0,\dots,n \big\}\in \F^w_k$$
  for $(A'_i,B'_i) := f\shift(\!\vr)(\vso) (A_i,B_i)$.

From Lemma~\ref{lem:shifting} we know that $\sigma'$ is a star. Since $(X,Y)$ emulates~$\vr$,%
   \COMMENT{}
   we have $\sigma'\sub \vS_k$ by~\eqref{eq:S2}.%
   \COMMENT{}
   It remains to show that $\big|\bigcap_{i=0}^n B'_i\big| < w$. As in Lem\-ma~\ref{lem:twd-shiftable}, we shall prove this by rewriting the intersection of all the~$B'_i$ as an intersection of the two sides of a suitable separation,%
   \COMMENT{}
   and use submodularity and the choice of~$(X,Y)$ to show that this separation has order~$<w$.

By~\eqref{eq:S2} and the definition of~$f\shift(\!\vr)(\vso)$, we have $(A'_0,B'_0) = (A_0\cup X,B_0\cap Y)$, while $(A'_i,B'_i) = (A_i\cap Y,B_i\cup X)$ for $i\ge 1$.
  Since the $(A_i,B_i)$ are separations, i.e.\ in~$\vU$, 
  so is $\big(\bigcap_{i=1}^n  B_i,\bigcup_{i=1}^n A_i\big)$.
  As trivially $(V,B_0)\in\vU$, this implies that, for $B^*:= \bigcap_{i=1}^n B_i$, also%
    \COMMENT{}
  $$\Big(\bigcap_{i=1}^n B_i \cap V\,,\, \bigcup_{i=1}^n A_i\cup B_0\Big) \specrel={\eqref{eq:S2}} (B^*,B_0) \in\,\vU\,.$$
Note that
\begin{equation}\label{eq:B}
\ord(B^*,B_0) = \abs{B^*\cap B_0} = \big|\bigcap_{i=0}^n B_i\big| < w
\end{equation}
since $\sigma\in\F^w_k$.

As $\vr\le (X,Y)$, and also $\vr\le (B^*,B_0)$ by~\eqref{eq:S2}, we further have
 $$\vr\le (X\cap B^*, Y\cup B_0) \le (X,Y)\le\vrdash.$$
 Hence if $\ord(X\cap B^*, Y\cup B_0) < \ord(X,Y)$ then this would contradict our choice of~$(X,Y)$.%
   \COMMENT{}
   Therefore $\ord(X\cap B^*, Y\cup B_0) \ge \ord(X,Y)$. As
 $$\ord(X\cap B^*,Y\cup B_0) + \ord(X\cup B^*, Y\cap B_0)\ \le\ \ord(X,Y)+\ord(B^*,B_0)$$
by submodularity, we deduce that
 $$\ord(X\cup B^*, Y\cap B_0)\ \le\ \ord(B^*,B_0)\ <\ w$$
by~\eqref{eq:B}. Hence
 $$\Big|\bigcap_{i=0}^n B'_i\,\Big| = \Big|(B_0\cap Y) \cap \bigcap_{i=1}^n (B_i\cup X)\Big| = \big| B_0\cap Y, B^*\cup X\big| < w\,, $$
which means that $\sigma'\in\F^w_k$ as desired.
\end{proof}

The following translation lemma is proved like Lemma~\ref{lem:twd-tree}:%
   \COMMENT{}

\begin{LEM}
  $G$ has an $S_k$-tree over $\F^w_k$ if and only if it has a \td{} of width~$ < w-1$ and adhesion~$ < k$.%
   \COMMENT{}
\end{LEM}

Theorem~\ref{thm:strong} and our two lemmas imply the following duality theorem:

\begin{THM}[Tangle-treewidth duality for bounded adhesion]~\\[1pt]
  The following assertions are equivalent for all finite graphs $G\ne\emptyset$ and integers ${w\ge k>0\!:}$
   \vskip-3pt\vskip0pt
  \begin{enumerate}[\rm (i)]\itemsep=0pt
  \item $G$ has an $\F^w_k$-tangle of~$S_k$.
  \item $G$ has no $S_k$-tree over $\F^w_k$.
  \item $G$ has no \td{} of width~$<w-1$ and adhesion~$<k$.\qed
  \end{enumerate}
\end{THM}

\ifArXiv

\section{Weakly Submodular Partition Functions}\label{sec:partitionsub}

Amini, Mazoit, Nisse and Thomass\'e~\cite{MazoitPartition}, and Lyaudet, Mazoit and Thomass\'e~\cite{MazoitPushing}, proposed a framework to unify duality theorems in graph minor theory which, unlike ours, is based exclusively on partitions. Their work, presented to us by Mazoit in the summer of 2013, inspired us to look for possible simplifications, for generalizations to separations that are not partitions, and for applications to tangle-like dense objects not covered by their framework. Our findings are presented in this paper and its sequel~\cite{DiestelOumDualityII}. Although our approach differs from theirs, we remain indebted to Mazoit and his coauthors for this inspiration.

Since the applications of our abstract duality theorem include the applications of~\cite{MazoitPartition}, it may seem unnecessary to ask whether our result also implies theirs directly. However, for completeness we address this question now.%
   \COMMENT{}

A~\emph{partition} of a finite set~$E$ is a set of disjoint subsets of~$E$, possibly empty, whose union is~$E$. We write $\P(E)$ for the set of all partitions of~$E$. In~\cite{MazoitPartition}, any function $\P(E)\to \mathbb{R}\cup\{\infty\}$ is called a \emph{partition function} of~$E$. We abbreviate $\psi(\{A_1,\dots,A_n\})$ to $\psi(A_1,\dots,A_n)$, but note that the partition remains unordered. A partition function $\psi$ is called \emph{weakly submodular} in~\cite{MazoitPartition} if, for every pair $(\A,\B)$ of partitions of~$E$ and every choice of $A_0\in\A$ and $B_0\in\B$, one of the following holds with $\A =: \{A_0,\ldots,A_n\}$ and $\B =: \{B_0,\ldots,B_m\}$:
\begin{enumerate}[(i)]
\item there exists a set $F$ such that $A_0\subseteq F\subseteq A_0\cup (E\sm B_0)$ and $\psi(A_0,\ldots,A_n)>\psi(F,A_1\sm F,\ldots,A_{n}\sm F)$;
\item $\psi(B_0,\ldots,B_m)\ge \psi\big(B_0\cup (E\sm A_0),B_1\cap A_0,\ldots,B_{m}\cap A_0\big)$.
\end{enumerate}

Let us translate this to our framework. Given $A\sub E$, let $\bar A:= E\sm A$. Then $\vU := \{ (A,\bar A) : A\sub E\}$ is a universe. Given a partition function~$\psi$ of~$E$, let\looseness=-1
 $$\vS_k = \big\{(A,\bar A)\in \vU :  \psi(A,\bar A) < k\big\}.$$%
   \COMMENT{}
Every partition $\A = \{A_0,\ldots,A_n\}$ defines a star $\{(A_0,\bar A_0),\dots, (A_n,\bar A_n)\}\sub\vU$, which we denote by~$\sigma(\A)$. Let
\[\F_k := \big\{ \sigma(\A) : \A\in\P(E)\text{ and } \psi(\A) < k \big\}\cup \{\{(\bar X,X)\}: \abs{X}\le 1,~ \psi(\bar X, X)<k\}.\]%
\COMMENT{}%
 If all the stars in $\F_k$ are subsets of~$\vS_k$, we call $\psi$ \emph{monotone}. All the  weakly submodular partition functions used in~\cite{MazoitPartition} for applications are monotone, and we do not know whether any exist that are not.%
   \COMMENT{}

\begin{LEM}\label{lem:weaksub-sep}
  If $\psi$ is weakly submodular, then $\vS_k$ is $\F_k$-separable.
\end{LEM}

\begin{proof}
  Let $\vS=\vS_k$. Let $\vr,\rvdash\in\vS$ that are not forced by~$\F_k$ and satisfy $\vr\le\vrdash$ be given. Pick $\vso = (X,Y)\in\vS$ with $\vr\le \vso\le\vrdash$ and $\psi(X,Y)$ minimum.
   We claim that $\vso$ emulates $\vr$ in~$\vS$ for~$\F_k$, and that $\svo$ emulates $\rvdash$ in~$\vS$ for~$\F_k$. By symmetry, it is enough to prove that $\vso$ emulates $\vr$ for~$\F_k$. 

  We first prove that $\vso$ emulates~$\vr$. Let $\vs=(A,B)\in \vSr\sm\{\rv\}$ be given. Since $\psi$ is weakly submodular, one of the following assertions holds:%
   \COMMENT{}
  \begin{enumerate}[\rm (i)]\itemsep=0pt
  \item there exists $F$ such that $Y\subseteq F\subseteq Y\cup B$ and $\psi(X,Y)>\psi(\bar F,F)$;%
     \COMMENT{}
  \item $\psi(A,B\ge \psi(A\cup X,B\cap Y)$.
  \end{enumerate}
  Since $A\cap X\sub \bar F\sub X$, we have $\vr\le (\bar F,F)\le \vrdash$. So (i) does not hold, by the choice of $(X,Y)$. So by (ii), $\vso\vee\vs =(A\cup X,B\cap Y)\in \vS$. This proves that $\vso$ emulates~$\vr$.%

  Now let us show that stars can be shifted. Let 
  \[\sigma=\{(A_i,B_i):i=0,\ldots,n\}\]
  be a star in~$\F_k\cap {\vSr}$, with $(A_0,B_0)\ge \vr$. 
  We have to show that
 \[\sigma' = \big\{(A'_i,B'_i): i=0,\dots,n \big\}\in \F_k\]
  for $(A'_i,B'_i) := f\shift(\!\vr)(\vso) (A_i,B_i)$.  Since $\vr\le (A_0,B_0)\le (B_i,A_i)$ for $i\ge 1$, we have
  $(A'_0,B'_0) = (A_0\cup X,B_0\cap Y)$, while $(A'_i,B'_i) = (A_i\cap Y,B_i\cup X)$ for $i\ge 1$.

 If $n=0$, then $\abs{B_0'}\le \abs{B_0}\le 1$ and so $\sigma'=\{(A_0',B_0')\}\in\F_k$.
 If $n\neq 0$, then $\psi(A_0,\ldots,A_n) < k$.
By the minimal choice of~$\vso=(X,Y)$, there exists no $F$ such that $X\subseteq F\subseteq X\cup A_0$ and $\psi(X,Y)>\psi(F,\bar F)$ (as earlier).
Applying the weak submodularity of $\psi$ with $(X,Y)$ and $(A_0,\ldots,A_n)$, we deduce that 
\[\psi(A'_0,\dots,A'_n) = \psi(A_0\cup X,\ A_1\cap Y,\,\ldots\,,A_n\cap Y)\le \psi(A_0,\ldots,A_n) < k.\]
Thus, $\sigma'\in \F_k$.
\end{proof}

In~\cite{MazoitPartition}, a \emph{$k$-bramble} for a weakly submodular partition function~$\psi$ of~$E$ is a non-empty set $\B$ of pairwise intersecting subsets of~$E$ that contains an element from every partition $\A$ of $E$ with ${\psi(\A) < k}$. It is \emph{non-principal} if it contains no singleton set~$\{e\}$. 
In our terminology, Amini et al.~\cite{MazoitPartition} prove that there exists a non-principal $k$-bramble for~$\psi$ if and only if there is no $S_k$-tree over~$\F_k$; they call this a `partitioning $k$-search tree'.

Now any $k$-bramble~$\B$ defines an orientation $O$ of~$S_k$: given $\{A,B\}\in S_k$ exactly one of $A,B$ must lie in~$\B$,%
   \COMMENT{}
  and if $B$ does we put $(A,B)$ in~$O$. Clearly $O$ is consistent%
   \COMMENT{}
   and avoids non-singleton stars in~$\F_k$, and if $\B$ is non-principal it avoids all singleton stars in $\F_k$. Conversely, given an orientation $O$ of~$S_k$, let ${\B := \{ B : (A,B)\in O\}}$. If $O$ is consistent, no two elements of~$\B$ are disjoint. If $O$ avoids singleton stars in~$\F_k$, then $\B$ is non-principal.%
   \COMMENT{}
   And finally, if $\psi$ is monotone and $O$ avoids~$\F_k$, then $\B$ contains an element from every partition $\A=\{A_1,\dots, A_n\}$ of~$E$ with $\psi(\A) < k$: since $\sigma(\A)\in\F_k$ there is $(A_i,\bar A_i)\in \sigma(\A)\sm O$, which means that $(\bar A_i,A_i)\in O$%
   \COMMENT{}
   and thus $A_i\in\B$.

Lemma~\ref{lem:weaksub-sep} and Theorem~\ref{thm:strong} thus imply the duality theorem of Amini et al.~\cite{MazoitPartition} for monotone weakly submodular partition functions:

\begin{THM}
  The following assertions are equivalent for all monotone weakly submodular partition functions $\psi$ of a finite set~$E$ and $k>0$:
  \begin{enumerate}[\rm (i)]\itemsep=0pt
  \item There exists a non-principal $k$-bramble for~$\psi$.
  \item $S_k$ has an $\F_k$-tangle.
  \item There exists no $S_k$-tree over $\F_k$.
  \item There exists no partitioning $k$-search tree.
  \end{enumerate}
\end{THM}%
   \COMMENT{}

\fi

\section{Further applications}\label{sec:further}

There are some obvious ways in which we can modify the sets $\F$ considered so far in this section to create new kinds of highly cohesive substructures and obtain associated duality theorems as corollaries of Theorem~\ref{thm:strong}. For example, we might strengthen the notion of a tangle by forbidding not just all the 3-sets of separations whose small sides together cover the entire graph or matroid, but forbid all such $m$-sets with $m$ up to some fixed value $n>3$. The resulting set $\F$ can then be replaced by its subset $\F^*$ of stars without affecting the set of consistent orientations avoiding~$\F$, just as in Lemma~\ref{T*toT}.%
   \COMMENT{}

Similarly, we might like \td s whose decomposition trees have degrees of at least~$n$ at all internal nodes: graphs with such a \td, of width and adhesion~$<k$ say, would `decay fast' along $(<k)$-separations. Such \td s can be described as $S_k$-trees over the subset of all $(\ge n)$-sets and singletons in the $\F_k$ defined in Section~\ref{sec:twd}.%
   \COMMENT{}

Another ingredient we might wish to change are the singleton stars in $\F$ associated with leaves. For example, we might be interested in \td s whose leaf parts are planar, while its internal parts need not be planar but might have to be small. Theorem~\ref{thm:strong} would offer dual objects also for such decompositions.

Conversely, it would be interesting to see whether other concrete highly cohesive substructures than those discussed in the preceding sections can be described as $\F$-tangles for some~$\F$ of a suitable set $S$ of separations~-- of a graph or something else.

Bowler~\cite{B14KkandNPC} answered this in the negative for complete minors in graphs, a natural candidate. Using the terminology of~\cite{DiestelBook16} for minors $H$ of~$G$, let us say that a separation $(A,B)$ of $G$ \emph{points to} an $I H\sub G$ if this $I H$ has a branch set in $B\sm A$ but none in~$A\sm B$. A~set of oriented separations \emph{points to} a given $IH$ if each of its elements does. Clearly, for every $IK_k\sub G$ exactly one of $(A,B)$ and $(B,A)$ in $\vS_k$ points to this~$IK_k$.

\begin{THM}{\rm\cite{B14KkandNPC}}
For every $k\ge 5$ there exists a graph $G$ such that for no set $\F\sub 2^{\vS_k}$ of stars are the $\F$-tangles of $S_k$ precisely the orientations of~$S_k$ that point to some $IK_k\sub G$.
\end{THM}

To prove this,%
   \COMMENT{}
   Bowler considered as $G$ a subdivision of $K_k$ obtained by subdividing every edge of $K_k$ exactly once. He constructed an orientation $O$ such that every star $\sigma\sub O$ points to an~$IK_k$ but the entire $O$ does not. This~$O$, then, avoids every $\F$ consisting only of stars not pointing to any~$IK_k$. But any $\F\sub 2^{\vS_k}$ such that the orientations of $S_k$ pointing to an $IK_k$ are precisely the $\F$-tangles must consist of stars not pointing to an~$IK_k$, since any star that does is contained in the unique orientation of~$S_k$ pointing to the $IK_k$ to which this star points.

However, $K_k$~minors can be captured by $\F$-avoiding orientations of $S_k$ if we do not insist that $\F$ contain only stars but allow it to contain \emph{weak stars}: sets of oriented separations that pairwise either cross or point towards each other (formally: consistent antichains in~$\vS_k$).%
   \COMMENT{}
   In~\cite{DiestelOumDualityII} we prove a duality theorem for orientations of separation systems avoiding such collections $\F$ of weak stars.

In~\cite{ProfileDuality}, we show that Theorem~\ref{thm:strong} implies duality theorems for $k$-blocks and for any given subset of $k$-tangles.

\bibliographystyle{abbrv}
\bibliography{collective}

\end{document}